\documentclass[12pt]{amsart}
\newtheorem{thm}{Theorem}
\newtheorem{prop}{Proposition}
\newtheorem{lemma}{Lemma}
\newtheorem{cor}{Corollary}
\newcommand{\im}{\operatorname{im}}
\newcommand{\Rho}{\mathrm{P}}
\newcommand{\End}{\mathop{\mathrm{End}}\nolimits}
\newcommand{\coker}{\operatorname{coker}}

\begin{document}
\title{Conformally Fedosov manifolds}
\author[Michael Eastwood]{Michael Eastwood}
\address{\hskip-\parindent
School of Mathematical Sciences\\
University of Adelaide\\ 
SA 5005\\ 
Australia}
\email{meastwoo@member.ams.org}
\author[Jan Slov\'ak]{Jan Slov\'ak}
\address{\hskip-\parindent
Department of Mathematics and Statistics\\
Masaryk University,\newline
611 37 Brno, Czech Republic}
\email{slovak@math.muni.cz}
\begin{abstract} We introduce the notion of a conformally Fedosov structure and
construct an associated Cartan connection. When an appropriate curvature
vanishes, this allows us to construct a family of natural differential 
complexes akin to the BGG complexes from parabolic geometry.
\end{abstract}
\renewcommand{\subjclassname}{\textup{2010} Mathematics Subject Classification}
\maketitle

\footnotetext{This research was supported by the Czech Grant Agency. The
authors would like to thank the Agency for their generous support under Grant
P201/12/G028.}
\section{Introduction}\label{intro}
On a Riemannian manifold there is a unique torsion-free connection on the
tangent bundle preserving the metric. Known as the {\em Levi Civita\/}
connection, it is the basis for doing calculus and understanding the geometry
on such a manifold. On a symplectic manifold, there is no preferred connection.
Instead, there are many {\em symplectic connections\/}, torsion-free
connections preserving the symplectic form. Choosing one defines what is known
as a {\em Fedosov manifold\/}~\cite{GRS}.

On a conformal manifold there is no preferred connection on the tangent bundle:
each metric in the conformal class gives rise to its own Levi-Civita
connection. Instead, a conformal structure induces a canonically defined {\em
Cartan connection\/}~\cite[\S1.6.7]{CS}. It is the fundamental object in
conformal differential geometry and can be regarded as a connection on an
auxiliary vector bundle~\cite{BEG}. Its curvature provides the basic conformal
invariant.

On a {\em projective\/} manifold~\cite[\S4.1.5]{CS}, there is an equivalence
class of torsion-free connections on the tangent bundle. Again, it is the
Cartan connection, built from these affine connections, which may equivalently
be regarded~\cite{BEG} as a connection on some auxiliary vector bundle and 
whose curvature is the basic projective invariant. 

On a {\em conformally symplectic\/} manifold~\cite{V} there is a local
equivalence class of symplectic forms defined only up to scale. In this article
we shall show that one can combine projective differential geometry with the
notion of a Fedosov manifold to obtain what we shall call {\em conformally
Fedosov\/} manifolds. They are obtained by adding further structure to a
conformally symplectic manifold and have the remarkable property that a 
canonical Cartan connection can then be constructed. This lies outside the 
realm of parabolic differential geometry~\cite{CS}. 

\subsection{Notation and terminology}
Notice that we are choosing to write `conformally Fedosov' rather than `locally
conformally Fedosov' or `locally conformal Fedosov.' These various alternatives
are regularly employed in the context of K\"ahler or symplectic geometry. Our
usage is chosen for several reasons. Firstly, we suppress the word `local' in
our terminology. In this we follow by analogy the standard convention that the
round metric on the sphere is `conformally flat' rather than being `locally
conformally flat,' for example. Secondly, our terminology is reasonably
succinct. Thirdly, for K\"ahler geometry the corresponding terminology of 
`conformally K\"ahler' was introduced by Westlake~\cite{W} already in 1954\@.

We shall often need to manipulate tensors on a smooth manifold and, for 
this purpose, we shall use Penrose's abstract index notation~\cite{OT}. In 
brief, covariant tensors will be decorated with subscripts, contravariant 
tensors with superscripts, and the natural pairing between vectors and 
$1$-forms by repeating an index $X^a\omega_a$ in accordance with the `Einstein 
summation convention.' For any tensor $\phi_{abc}$ we shall write 
$\phi_{(abc)}$ for its symmetric part and $\phi_{[abc]}$ for its skew part. 
For example, to say that $\omega_{ab}$ is a $2$-form is to say that 
$\omega_{ab}=\omega_{[ab]}$ or, equivalently, that $\omega_{(ab)}=0$ and, 
for any torsion-free connection~$\nabla_a$, the expressions
$$\nabla_{[a}\omega_{bc]}\quad\mbox{and}\quad
X^a\nabla_a\omega_{bc}-2(\nabla_{[b}X^a)\omega_{c]a}$$
deliver the exterior derivative of $\omega_{ab}$ and the Lie derivative of
$\omega_{ab}$ in the direction of the vector field~$X^a$, respectively.

We shall write $\Lambda^p$ for the bundle of $p$-forms on a smooth manifold, 
suppressing the name of the manifold itself. The exterior derivative will be 
denoted by $d:\Lambda^p\to\Lambda^{p+1}$.

\section{Conformally symplectic manifolds}
In the first instance, a {\em conformally symplectic\/} manifold~\cite{B,V} is
an even-dimensional manifold $M$ of dimension at least four equipped with a
non-degenerate $2$-form $J$ such that
\begin{equation}\label{def_alpha}dJ=2\alpha\wedge J\end{equation}
for some closed $1$-form~$\alpha$. Non-degeneracy of $J$ ensures
$J:\Lambda^1\to\Lambda^3$ is injective whence $\alpha$ is uniquely determined
by $J$, should such a $1$-form exist. It is called the {\em Lee form\/}
\cite{L} and, in case $\dim M\geq 6$ we see that
$$0=d^2J=2d\alpha\wedge J+2\alpha\wedge dJ
=2d\alpha\wedge J+4\alpha\wedge\alpha\wedge J=
2d\alpha\wedge J$$
and, as $J\wedge\underbar\enskip:\Lambda^2\to\Lambda^4$ is injective, closure
of $\alpha$ is automatic. If we rescale $J$ by a positive smooth function, say
$\hat J=\Omega^2J$, then (\ref{def_alpha}) remains valid with $\alpha$ replaced
by $\hat\alpha=\alpha+\Upsilon$ for $\Upsilon\equiv d\log\Omega$. Hence, the
notion of conformally symplectic is invariant under such rescalings (and also
in dimension $4$ since $d\Upsilon=0$). Locally, we may use this freedom to
eliminate $\alpha$ and obtain an ordinary symplectic structure. Globally,
however, this need not be the case. For example, the rescaled symplectic form
$$J\equiv\left(1/{\|x\|}\right)^2(dx^1\wedge dx^2+dx^3\wedge dx^4+\cdots)$$
on ${\mathbb{R}}^{2n}$ is invariant under dilation $x\mapsto\lambda x$ and,
therefore, descends to a conformally symplectic structure on $S^1\times
S^{2n-1}$ whereas there is no global symplectic form on this manifold. 

More precisely, a {\em conformally symplectic\/} manifold is a pair $(M,[J])$
where $[J]$ is an equivalence class of non-degenerate $2$-forms satisfying
(\ref{def_alpha}) where $J$ and $\hat J$ are said to be equivalent if and only
if $\hat J=\Omega^2J$ for some positive smooth function~$\Omega$. As one often
does in conformal geometry in which a Riemannian metric is only defined up to
local rescaling $g\mapsto\hat g=\Omega^2g$, it is usual to pick a
representative $J$ and work with that representative, whilst checking that
one's conclusions are independent of this choice. The basic example of this
approach is in noting that the requirement (\ref{def_alpha}) is itself
independent of such a choice.

\section{Projective manifolds}
A {\em projective structure\/}~\cite{E1} on a manifold $M$ is an equivalence
class of torsion-free affine connections on $M$, where two connections
$\nabla_a$ and $\hat\nabla_a$ are said to be projectively equivalent if and
only if
\begin{equation}\label{projectivechange}
\hat\nabla_a\phi_b=\nabla_a\phi_b-\nu_a\phi_b-\nu_b\phi_a\end{equation}
for some $1$-form~$\nu_a$. 
\begin{lemma}\label{one} If $J_{ab}$ is skew, then
$$\hat\nabla_{(a}J_{b)c}=\nabla_{(a}J_{b)c}-3\nu_{(a}J_{b)c}.$$
\end{lemma}
\begin{proof}
The Leibniz rule extends (\ref{projectivechange}) 
to all other tensors. Thus,
$$\hat\nabla_aJ_{bc}=\nabla_aJ_{bc}-2\nu_aJ_{bc}-\nu_bJ_{ac}-\nu_cJ_{ba}$$
and symmetrising over $ab$ gives the desired conclusion.
\end{proof}
\begin{prop}\label{invariance} If $J_{ab}$ is skew, then the requirement that
\begin{equation}\label{secondeq}\nabla_{(a}J_{b)c}=\beta_{(a}J_{b)c}
\end{equation}
for some $1$-form $\beta_a$ is projectively invariant. 
\end{prop}
\begin{proof} {From} Lemma~\ref{one}, if $\nabla_a$ is replaced
by~$\hat\nabla_a$ according to (\ref{projectivechange}), then (\ref{secondeq})
still holds but with $\beta_a$ replaced by $\hat\beta_a\equiv\beta_a-3\nu_a$.
\end{proof}

\section{Conformally Fedosov manifolds}
Let $(M,[J])$ be a conformally symplectic manifold. We may express the
requirement (\ref{def_alpha}) in terms of any torsion-free connection
$\nabla_a$ as
\begin{equation}\label{firsteq}\nabla_{[a}J_{bc]}=2\alpha_{[a}J_{bc]}.
\end{equation}
Now let us also insist on~(\ref{secondeq}). As observed in
Proposition~\ref{invariance}, this is only a restriction on the projective
class of~$\nabla_a$. We may assemble these conditions into the following formal 
definition. A {\em conformally Fedosov\/} manifold is a triple 
$(M,[J],[\nabla])$ where 
\begin{itemize}
\item $M$ is a smooth manifold of dimension $2n\geq 4$,
\item $[J]$ is an equivalence class of non-degenerate $2$-forms defined up to 
rescaling $J\mapsto\hat J=\Omega^2J$ for some positive function~$\Omega$,
\item $[\nabla]$ is a projective structure, i.e.~an equivalence class of
torsion-free connections defined up to (\ref{projectivechange}) for 
some $1$-form~$\nu_a$,
\item the following equations hold
\begin{equation}\label{eqns}
\nabla_{[a}J_{bc]}=2\alpha_{[a}J_{bc]}\qquad
\nabla_{[a}\alpha_{b]}=0\qquad\nabla_{(a}J_{b)c}=\beta_{(a}J_{b)c},
\end{equation}
for some $1$-forms $\alpha_a$ and $\beta_a$. 
\end{itemize}
We have already observed that the requirement (\ref{firsteq}) depends only on
the conformal class of~$J_{ab}$ (and that $\nabla_{[a}\alpha_{b]}=0$ is
automatic for $n\geq3$). Proposition~\ref{invariance} says that
(\ref{secondeq}) is projectively invariant. Finally, to make sure that this
definition makes sense, let us observe that (\ref{secondeq}) is also
conformally invariant: if $\hat J_{ab}=\Omega^2J_{ab}$, then (\ref{secondeq})
continues to hold but with $\beta_a$ replaced by
$\hat\beta_a=\beta_a+2\Upsilon_a$.

We shall often have occasion to `raise and lower' indices using $J_{ab}$ and
its inverse~$J^{ab}$. Specifically, let $J_{ac}J^{bc}=\delta_a{}^b$, where
$\delta_a{}^b$ is the Kronecker delta. We then decree that 
$$\phi^a\equiv J^{ab}\phi_b\qquad\psi_b\equiv\psi^aJ_{ab}$$
and henceforth freely make use of these options without comment.

\begin{prop} Let $(M,[J],[\nabla])$ be a conformally Fedosov manifold. Any
representatives $J_{ab}$ and\/ $\nabla_a$ of the structure uniquely
determine the\/ $1$-forms $\alpha_a$ and $\beta_a$ occurring in 
{\rm(\ref{eqns})} and, conversely,
\begin{equation}\label{nablaJ}\textstyle\nabla_aJ_{bc}=
2\alpha_{[a}J_{bc]}+\frac23\beta_{(a}J_{b)c}-\frac23\beta_{(a}J_{c)b}
\end{equation}
determines the full covariant derivative $\nabla_aJ_{bc}$.
\end{prop}
\begin{proof} Let $J^{ab}$ denote the inverse of $J_{ab}$. Then 
the identities
$$3J^{bc}\nabla_{[a}J_{bc]}=2(n-2)\alpha_a\quad\mbox{and}\quad
2J^{bc}\nabla_{(a}J_{b)c}=(2n+1)\beta_a$$
readily follow from (\ref{eqns}). Conversely, expanding the right hand side of
$$\textstyle
2\alpha_{[a}J_{bc]}+\frac23\beta_{(a}J_{b)c}-\frac23\beta_{(a}J_{c)b}
=\nabla_{[a}J_{bc]}+\frac23\nabla_{(a}J_{b)c}-\frac23\nabla_{(a}J_{c)b}$$
gives $\nabla_aJ_{bc}$, as required.\end{proof}
\begin{prop}\label{making_the_mastereq}
For any conformally Fedosov manifold $(M,[J],[\nabla])$, if a 
representative $2$-form $J_{ab}$ is chosen, then there is a unique 
torsion-free connection in the projective class such that 
\begin{equation}\label{mastereq}\nabla_aJ_{bc}=2J_{a[b}\alpha_{c]}.
\end{equation}
\end{prop}
\begin{proof}
When the connection $\nabla_a$ is replaced by $\hat\nabla_a$ according to
(\ref{projectivechange}), the $1$-form $\alpha_a$ does not change but $\beta_a$
is replaced by $\hat\beta_a=\beta_a-3\nu_a$, Therefore, we can uniquely arrange
that $\alpha_a+\beta_a=0$, in which case (\ref{nablaJ}) implies that 
$$\textstyle\nabla_aJ_{bc}=
2\alpha_{[a}J_{bc]}-\frac23\alpha_{(a}J_{b)c}+\frac23\alpha_{(a}J_{c)b}$$
and expanding the right hand side gives $2J_{a[b}\alpha_{c]}$, as required.
\end{proof}
In view of this Proposition, an alternative definition of a conformally Fedosov
manifold is as follows. Firstly, define an equivalence relation on pairs
$(J,\nabla)$ consisting of a non-degenerate symplectic form $J_{ab}$ and a 
torsion-free connection $\nabla_a$ by allowing simultaneous replacements 
\begin{equation}\label{simultaneousrescaling}
\begin{array}{rcl}J_{ab}&\!\!\mapsto\!\!&
\hat J_{ab}=\Omega^2J_{ab}\\
\nabla_a\phi_b&\!\!\mapsto\!\!&
\hat\nabla_a\phi_b=\nabla_a\phi_b-\Upsilon_a\phi_b-\Upsilon_b\phi_a,
\end{array}\enskip
\mbox{where }\Upsilon_a=\nabla_a\log\Omega.
\end{equation}
Writing $[J,\nabla]$ for the equivalence class of such pairs, a conformally
Fedosov manifold may then be defined as a pair $(M,[J,\nabla])$ such that
(\ref{mastereq}) holds (and one can check directly that (\ref{mastereq}) is 
invariant under (\ref{simultaneousrescaling}) if one decrees that 
$\alpha_a\mapsto\hat\alpha_a=\alpha_a+\Upsilon_a$).
For the rest of this article we shall adopt this alternative definition of a 
conformally Fedosov manifold. By analogy with ordinary conformal structures, 
we shall refer to the pair $[J,\nabla]$ as a {\em conformal class\/}. 

\begin{prop} Any conformally symplectic manifold $(M,[J])$ can be extended to 
a conformally Fedosov structure $(M,[J,\nabla])$.
\end{prop}
\begin{proof} Pick a representative $2$-form~$J_{ab}$. We are required to find
a torsion-free connection $\nabla_a$ such that (\ref{mastereq}) is satisfied
for some $1$-form~$\alpha_a$. Recall that the $1$-form $\alpha_a$ is already
determined by (\ref{firsteq}) independent of choice of~$\nabla_a$. Locally, 
there is no problem in finding a suitable~$\nabla_a$: choose 
$\Omega$ such that $\hat J_{ab}=\Omega^2J_{ab}$ is closed and define 
$\nabla_a$ by (\ref{simultaneousrescaling}) where $\hat\nabla_a$ is the flat 
connection in Darboux co\"ordinates for $\hat J_{ab}$. We may use a 
partition of unity to patch these connections together.
\end{proof}

\begin{prop}\label{finite_type} Equation\/ {\rm(\ref{mastereq})} is 
equivalent to 
\begin{equation}\label{newmastereq}
\nabla_aJ^{bc}=2\alpha^{[b}\delta_a{}^{c]},\end{equation}
where $\alpha^b\equiv J^{bc}\alpha_c$.
\end{prop}
\begin{proof} Recall that $J_{bc}J^{bd}=\delta_c{}^d$. Differentiating this and 
substituting from (\ref{mastereq}) gives
$$0=J_{bc}\nabla_aJ^{bd}+2J_{a[b}\alpha_{c]}J^{bd}=
J_{bc}\nabla_aJ^{bd}-\delta_a{}^d\alpha_c+J_{ac}\alpha^d.$$
Therefore, 
$$\nabla_aJ^{bc}=J^{be}J_{de}\nabla_aJ^{dc}=
J^{be}\left(\delta_a{}^c\alpha_e-J_{ae}\alpha^c\right)=
\delta_a{}^c\alpha^b-\delta_a{}^b\alpha^c,$$
as required.
\end{proof}
\begin{cor} A projective structure $[\nabla]$ cannot necessarily be extended 
to a conformally Fedosov structure. 
\end{cor}
\begin{proof}
That equation (\ref{newmastereq}) hold for some vector field $\alpha^a$ is 
equivalent to requiring that 
$$\mbox{the trace-free part of }(\nabla_aJ^{bc})=0,$$
which is a system of finite type as explained in~\cite{BCEG}. Hence, there are
obstructions to its solution (and writing it as (\ref{newmastereq}) is the
first step in its prolongation).
\end{proof}

\section{Curvature}
For any torsion-free affine connection~$\nabla_a$, the curvature
$R_{ab}{}^c{}_d$ of $\nabla_a$ is characterised by the equation
\begin{equation}\label{def_of_curvature}
(\nabla_a\nabla_b-\nabla_b\nabla_a)X^c=R_{ab}{}^c{}_dX^d.\end{equation}
Recall that we are free to `lower an index' and write the curvature as 
$R_{abcd}$ in the presence of a non-degenerate $2$-form~$J_{ab}$.

\begin{thm}\label{curvature_decomposition}
Choosing any representative $2$-form $J_{ab}$ and connection
$\nabla_a$ of a conformally Fedosov manifold\/ $(M,[J,\nabla])$, the curvature
$R_{ab}{}^c{}_d$ of\/ $\nabla_a$ may be uniquely written as
$$R_{ab}{}^c{}_d=W_{ab}{}^c{}_d+\delta_a{}^c\Rho_{bd}-\delta_b{}^c\Rho_{ad},$$
where $\Rho_{ab}$ is a symmetric tensor and $W_{ab}{}^c{}_d$ satisfies
\begin{equation}\label{ttf}W_{ab}{}^c{}_d=W_{[ab]}{}^c{}_d\qquad
W_{[ab}{}^c{}_{d]}=0\qquad W_{ab}{}^a{}_d=0.\end{equation}
Under conformal rescaling {\rm(\ref{simultaneousrescaling})}, the tensor 
$W_{ab}{}^c{}_d$ is unchanged whilst
$$\hat\Rho_{ab}=\Rho_{ab}-\nabla_a\Upsilon_b+\Upsilon_a\Upsilon_b.$$
Furthermore, the tensor $W_{abcd}$ may be uniquely decomposed as 
$$\textstyle W_{abcd}=
V_{abcd}-\frac{3}{2n-1}J_{ac}\Phi_{bd}+\frac{3}{2n-1}J_{bc}\Phi_{ad}+
J_{ad}\Phi_{bc}-J_{bd}\Phi_{ac}+2J_{ab}\Phi_{cd},$$
where 
\begin{equation}\label{whatVsatisfies}
V_{abcd}=V_{[ab](cd)}\qquad V_{[abc]d}=0\qquad J^{ab}V_{abcd}=0\end{equation}
and\/ $\Phi_{ab}$ is symmetric. 
\end{thm}
\begin{proof}
The curvature of any torsion-free connection may be uniquely and conveniently 
written as
\begin{equation}\label{projectivedecomposition}
R_{ab}{}^c{}_d=
W_{ab}{}^c{}_d+2\delta_{[a}{}^c\Rho_{b]d}+\beta_{ab}\delta^c{}_d\end{equation}
where $W_{ab}{}^c{}_d$ satisfies (\ref{ttf}) and $\beta_{ab}=-2\Rho_{[ab]}$. 
Let us suppose, for the moment, that $\nabla_aJ_{bc}=0$. Then, together with 
the Bianchi identity, we have
\begin{equation}\label{fedosovsymmetries}
R_{abcd}=R_{[ab](cd)}\quad\mbox{and}\quad R_{[abc]d}=0,\end{equation}
corresponding to an irreducible representation of 
${\mathrm{SL}}(2n,{\mathbb{R}})$. Branching this representation under 
${\mathrm{Sp}}(2n,{\mathbb{R}})\subset{\mathrm{SL}}(2n,{\mathbb{R}})$ gives
\begin{equation}\label{branched}R_{abcd}=V_{abcd}
+J_{ac}\Phi_{bd}-J_{bc}\Phi_{ad}+J_{ad}\Phi_{bc}-J_{bd}\Phi_{ac}
+2J_{ab}\Phi_{cd},\end{equation}
where $V_{abcd}$ satisfies (\ref{whatVsatisfies}) and $\Phi_{ab}$ is symmetric.
{From} (\ref{projectivedecomposition}) we see that
$$J^{cd}R_{abcd}=
J^{cd}\big(W_{abcd}+J_{ac}\Rho_{bd}-J_{bc}\Rho_{ad}-\beta_{ab}J_{cd}\big)=
-(2n+1)\beta_{ab}$$
whereas (\ref{fedosovsymmetries}) implies that $J^{cd}R_{abcd}$ should vanish. 
Therefore $\beta_{ab}=0$ and consequently $\Rho_{ab}$ is symmetric.
Thus, we have 
\begin{equation}\label{two_decompositions}
\begin{array}{rcl}R_{abcd}&\!=\!&W_{abcd}+J_{ac}\Rho_{bd}-J_{bc}\Rho_{ad}\\[3pt]
&\!=\!&V_{abcd}
+J_{ac}\Phi_{bd}-J_{bc}\Phi_{ad}+J_{ad}\Phi_{bc}-J_{bd}\Phi_{ac}
+2J_{ab}\Phi_{cd}\end{array}\end{equation}
from (\ref{projectivedecomposition}) and (\ref{branched}), respectively. Now
computing $J^{bc}R_{abcd}$ from each of these two decompositions gives
$(2n-1)\Rho_{ad}=2(n+1)\Phi_{ad}$. Substituting back and rearranging the result
gives the decomposition of $W_{abcd}$ as in the statement of the theorem.

This was all under the assumption that $\nabla_aJ_{bc}=0$ and locally, there is
always a connection $\nabla_a$ and $2$-form $J_{ab}$ in the conformal class
$[J,\nabla]$ for which this assumption is valid. In general, we must see how
our conclusions are affected by a conformal
change~(\ref{simultaneousrescaling}). The decomposition
(\ref{projectivedecomposition}) is familiar from projective differential
geometry~\cite{E1} and, since (\ref{simultaneousrescaling}) is controlled by a
{\em closed\/} $1$-form~$\Upsilon_a$, we have $\hat\beta_{ab}=0$ whilst 
$$\hat W_{ab}{}^c{}_d=W_{ab}{}^c{}_d\quad\mbox{and}\quad
\hat\Rho_{ab}=\Rho_{ab}-\nabla_a\Upsilon_b+\Upsilon_a\Upsilon_b.$$
Finally, having a lowered index, we see that $\hat W_{abcd}=\Omega^2W_{abcd}$
and the algebraic decomposition of $W_{abcd}$ given in the statement of the 
theorem remains valid with 
$$\hat V_{abcd}=\Omega^2 V_{abcd}\quad\mbox{and}\quad
\hat\Phi_{ab}=\Phi_{ab}.$$
This completes the proof.
\end{proof}
On a conformally Fedosov manifold, although $J_{ab}$ is only defined up to
scale, the local stipulation that $\nabla_aJ_{bc}=0$ for some torsion-free
connection $\nabla_a$ in the projective class characterises a globally defined
affine connection whose curvature decomposes as~(\ref{branched}) (also
depending only on $J_{ab}$ up to scale). More generally, the proof
of Theorem~\ref{curvature_decomposition} decomposes the curvature into three
${\mathrm{Sp}}(2n,{\mathbb{R}})$-irreducible parts,
$$V_{abcd}\in
\begin{picture}(72,5)
\put(4,1.5){\line(1,0){30}}
\put(4,1.2){\makebox(0,0){$\bullet$}}
\put(16,1.2){\makebox(0,0){$\bullet$}}
\put(28,1.2){\makebox(0,0){$\bullet$}}
\put(43,1.2){\makebox(0,0){$\cdots$}}
\put(50,1.5){\line(1,0){6}}
\put(56,1.2){\makebox(0,0){$\bullet$}}
\put(56,0.5){\line(1,0){12}}
\put(56,2.5){\line(1,0){12}}
\put(62,1.5){\makebox(0,0){$\langle$}}
\put(68,1.2){\makebox(0,0){$\bullet$}}
\put(4,8){\makebox(0,0){$\scriptstyle 2$}}
\put(16,8){\makebox(0,0){$\scriptstyle 1$}}
\put(28,8){\makebox(0,0){$\scriptstyle 0$}}
\put(56,8){\makebox(0,0){$\scriptstyle 0$}}
\put(68,8){\makebox(0,0){$\scriptstyle 0$}}
\end{picture}
\quad\enskip\Phi_{ab}\in
\begin{picture}(72,5)
\put(4,1.5){\line(1,0){30}}
\put(4,1.2){\makebox(0,0){$\bullet$}}
\put(16,1.2){\makebox(0,0){$\bullet$}}
\put(28,1.2){\makebox(0,0){$\bullet$}}
\put(43,1.2){\makebox(0,0){$\cdots$}}
\put(50,1.5){\line(1,0){6}}
\put(56,1.2){\makebox(0,0){$\bullet$}}
\put(56,0.5){\line(1,0){12}}
\put(56,2.5){\line(1,0){12}}
\put(62,1.5){\makebox(0,0){$\langle$}}
\put(68,1.2){\makebox(0,0){$\bullet$}}
\put(4,8){\makebox(0,0){$\scriptstyle 2$}}
\put(16,8){\makebox(0,0){$\scriptstyle 0$}}
\put(28,8){\makebox(0,0){$\scriptstyle 0$}}
\put(56,8){\makebox(0,0){$\scriptstyle 0$}}
\put(68,8){\makebox(0,0){$\scriptstyle 0$}}
\end{picture}
\quad\enskip\Rho_{ab}\in
\begin{picture}(72,5)
\put(4,1.5){\line(1,0){30}}
\put(4,1.2){\makebox(0,0){$\bullet$}}
\put(16,1.2){\makebox(0,0){$\bullet$}}
\put(28,1.2){\makebox(0,0){$\bullet$}}
\put(43,1.2){\makebox(0,0){$\cdots$}}
\put(50,1.5){\line(1,0){6}}
\put(56,1.2){\makebox(0,0){$\bullet$}}
\put(56,0.5){\line(1,0){12}}
\put(56,2.5){\line(1,0){12}}
\put(62,1.5){\makebox(0,0){$\langle$}}
\put(68,1.2){\makebox(0,0){$\bullet$}}
\put(4,8){\makebox(0,0){$\scriptstyle 2$}}
\put(16,8){\makebox(0,0){$\scriptstyle 0$}}
\put(28,8){\makebox(0,0){$\scriptstyle 0$}}
\put(56,8){\makebox(0,0){$\scriptstyle 0$}}
\put(68,8){\makebox(0,0){$\scriptstyle 0$}}
\end{picture}$$
according to
\begin{equation}\label{fullcurvaturedecomposition}
\textstyle R_{abcd}=V_{abcd}
+2J_{ab}\Phi_{cd}-2\Phi_{c[a}J_{b]d}+\frac{6}{2n-1}J_{c[a}\Phi_{b]d}
-2J_{c[a}\Rho_{b]d}\end{equation}
and under conformal change~(\ref{simultaneousrescaling}), we have
$$\hat V_{abcd}=\Omega^2V_{abcd}\qquad
\hat\Phi_{ab}=\Phi_{ab}\qquad
\hat\Rho_{ab}=\Rho_{ab}-\nabla_a\Upsilon_b+\Upsilon_a\Upsilon_b.$$
It is easy to give explicit formul{\ae} for these parts, viz.:--
$$\textstyle\Rho_{bd}=\mbox{\large$\frac{1}{2n-1}$}J^{ac}R_{abcd}\qquad
\Phi_{cd}=\mbox{\large$\frac{2n-1}{8(n+1)(n-1)}$}
\left(J^{ab}R_{abcd}-2\Rho_{cd}\right)$$
and $V_{abcd}$ is then determined by~(\ref{fullcurvaturedecomposition}). As an 
example, the curvature of ${\mathbb{CP}}_n$ with its standard Fubini-Study 
metric is given by 
$$R_{abcd}=g_{bd}J_{ac}-g_{ad}J_{bc}-g_{ac}J_{bd}+g_{bc}J_{ad}+2J_{ab}g_{cd}$$
and one easily computes that 
$$\Rho_{ab}=\mbox{\large$\frac{2(n+1)}{2n-1}$}g_{ab}\qquad\Phi_{ab}=g_{ab}
\qquad V_{abcd}=0.$$

As in the proof of Theorem~\ref{curvature_decomposition}, it is often
convenient locally to work in a gauge in which $\alpha_a=0$ for then
$\nabla_aJ_{bc}=0$ and the curvature $R_{abcd}$ decomposes according 
to~(\ref{branched}). Also recall from (\ref{two_decompositions}) that 
\begin{equation}\label{Rho_v_Phi}(2n-1)\Rho_{ab}=2(n+1)\Phi_{ab}.\end{equation}
We shall refer to a choice of pair $(J_{ab},\nabla_a)$ from a conformally
Fedosov structure $[J_{ab},\nabla_a]$ for which $\nabla_aJ_{bc}=0$ as a {\em 
Fedosov gauge\/}. This is in accordance with the notion of Fedosov 
manifold~\cite{GRS}. We pause here to examine some consequences of the Bianchi 
identity $\nabla_{[e}R_{ab]cd}=0$. {From} (\ref{branched}) we conclude that
$$\begin{array}{rcl}0&=&3J^{de}\nabla_{[e}R_{ab]cd}\\
&=&\nabla^dV_{abcd}+J_{ac}\nabla^d\Phi_{bd}-J_{bc}\nabla^d\Phi_{ad}
+\nabla_a\Phi_{bc}-\nabla_b\Phi_{ac}\\
&&\quad{}-\nabla_a\Phi_{bc}-2n\nabla_b\Phi_{ac}-\nabla_a\Phi_{bc}
+\nabla_b\Phi_{ac}+2n\nabla_a\Phi_{bc}+\nabla_b\Phi_{ac}\\
&&\qquad{}+2J_{ab}\nabla^d\Phi_{cd}+2\nabla_a\Phi_{bc}-2\nabla_b\Phi_{ac}\\
&=&\nabla^dV_{abcd}+J_{ac}\nabla^d\Phi_{bd}-J_{bc}\nabla^d\Phi_{ad}
+2J_{ab}\nabla^d\Phi_{cd}\\
&&\quad{}+(2n+1)\nabla_a\Phi_{bc}-(2n+1)\nabla_b\Phi_{ac}.
\end{array}$$
This suggests that one introduce the tensor
$$\textstyle Y_{abc}\equiv
\nabla_a\Phi_{bc}-\nabla_b\Phi_{ac}
+\frac{1}{2n+1}\big(J_{ac}\nabla^d\Phi_{bd}-J_{bc}\nabla^d\Phi_{ad}
+2J_{ab}\nabla^d\Phi_{cd}\big),$$
noting that
$$Y_{abc}=Y_{[ab]c}\qquad Y_{[abc]}=0\qquad J^{ab}Y_{abc}=0.$$
We have established the contracted Bianchi identity
\begin{equation}\label{contractedBianchi}
\nabla^dV_{abcd}+(2n+1)Y_{abc}=0.\end{equation}
For later use, it is convenient to introduce the tensor 
$S_a\equiv\frac{1}{2n+1}\nabla^b\Phi_{ab}$ so that
\begin{equation}\label{CottonYork}Y_{abc}=\nabla_a\Phi_{bc}-\nabla_b\Phi_{ac}
+J_{ac}S_b-J_{bc}S_a+2J_{ab}S_c.\end{equation}

\section{The tractor connection in conformal geometry}
Here we review the construction of the conformal tractor bundle and its
connection following the conventions of~\cite{BEG,E1}. We omit all details. The
purpose of this section is to establish notation and to motivate the
corresponding construction in the conformally Fedosov setting.

Firstly, we recall that the bundle $\Lambda^0[w]$ of {\em conformal densities
of weight\/} $w$ is defined as the trivial bundle $\Lambda^0$ in the presence
of a chosen metric $g_{ab}$ in the conformal class $[g_{ab}]$. Its smooth
sections may then be identified as smooth functions but if a different metric
is chosen, say $\hat{g}_{ab}=\Omega^2g_{ab}$, then the corresponding functions
are obliged to change according to $\hat\sigma=\Omega^w\sigma$. A similar
notion applies to tensors and tensor bundles. In particular, a $1$-form of
weight $w$ transforms according to $\hat\omega_a=\Omega^w\omega_a$ when
$g_{ab}$ is replaced by $\hat{g}_{ab}=\Omega^2g_{ab}$ and the corresponding 
bundle is denoted~$\Lambda^1[w]$. Tautologically, the
conformal metric itself is a symmetric covariant $2$-form of conformal
weight~$2$. The {\em standard tractor bundle\/} ${\mathbb{T}}$ is defined in
the presence of a chosen metric $g_{ab}$ to be the direct sum 
$${\mathbb{T}}=\Lambda^0[1]\oplus\Lambda^1[1]\oplus\Lambda^0[-1]$$
but if the metric is rescaled as $\hat{g}_{ab}=\Omega^2g_{ab}$, then this 
decomposition is mandated to change according to
$$\left[\begin{array}c\hat\sigma\\ \hat\mu_b\\ \hat\rho\end{array}\right]=
\left[\begin{array}c\sigma\\ \mu_b+\Upsilon_b\sigma\\ 
\rho-\Upsilon^b\mu_b-\frac12\Upsilon^b\Upsilon_b\sigma\end{array}\right],
\enskip\mbox{where }\Upsilon_a\equiv\nabla_a\log\Omega.$$
For a chosen metric $g_{ab}$ in the conformal class, the 
{\em tractor connection\/} is defined by 
$$\nabla_a\left[\begin{array}c\sigma\\ \mu_b\\
\rho\end{array}\right]= \left[\begin{array}c\nabla_a\sigma-\mu_a\\
\nabla_a\mu_b+g_{ab}\rho+\Rho_{ab}\sigma\\ 
\nabla_a\rho-\Rho_a{}^b\mu_b\end{array}\right],$$
where $\nabla_a\mu_b$ is the Levi-Civita connection of~$g_{ab}$. One checks 
that this definition is conformally invariant. As detailed in~\cite{BEG}, this 
construction is essentially due to T.Y.~Thomas~\cite{T} and is equivalent to 
the Cartan connection~\cite{C} constructed three years earlier.

\section{A conformally Fedosov tractor connection}\label{tractors}
Firstly, we shall build a {\em tractor bundle\/} on a conformally Fedosov
manifold, a vector bundle which we shall then endow with a canonically defined
connection. As usual, given a conformally Fedosov manifold $(M,[J,\nabla])$,
definitions will be given in terms of chosen representatives $J_{ab}$ and
$\nabla_a$ and then we shall check that these definitions respect the allowed
freedom~(\ref{simultaneousrescaling}). Firstly, we define the bundle
$\Lambda^0[w]$ of {\em conformal densities of weight\/} $w$ as the trivial
bundle in the presence of chosen representatives $(J_{ab},\nabla_a)$ but, under
the allowed replacements (\ref{simultaneousrescaling}), its sections regarded
as functions are decreed to change by $\hat\sigma=\Omega^w\sigma$.

For chosen representatives, the vector
bundle ${\mathbb{T}}$ is defined as
$${\mathbb{T}}=\Lambda^0[1]\oplus\Lambda^1[1]\oplus\Lambda^0[-1]$$
but this splitting is decreed to change as
\begin{equation}\label{changeofsplitting}
\left[\begin{array}c\hat\sigma\\ \hat\mu_b\\ \hat\rho\end{array}\right]=
\left[\begin{array}c\sigma\\ \mu_b+\Upsilon_b\sigma\\ 
\rho-\Upsilon^b\mu_b+\Upsilon^b\alpha_b\sigma\end{array}\right]\end{equation}
under~(\ref{simultaneousrescaling}), where $\alpha_a$ is defined 
by~(\ref{mastereq}). We may check this decree is self-consistent as follows.
$$\begin{array}{ccl}\hat{\hat\sigma}&=&\hat\sigma\;=\;\sigma\,,\\ 
\hat{\hat\mu}_b&=&\hat\mu_b+\hat\Upsilon_b\hat\sigma\;=\;
\mu_b+\Upsilon_b\sigma+\hat\Upsilon_b\sigma\;=\;
\mu_b+(\Upsilon_b+\hat\Upsilon_b)\sigma\,,\\
\hat{\hat\rho}&=&\hat\rho-\hat\Upsilon^b\hat\mu_b+
\hat\Upsilon^b\hat\alpha_b\hat\sigma\\
&=&\enskip\rho-\Upsilon^b\mu_b+\Upsilon^b\alpha_b\sigma
-\hat\Upsilon^b(\mu_b+\Upsilon_b\sigma)+
\hat\Upsilon^b(\alpha_b+\Upsilon_b)\sigma\\
&=&\quad\rho-(\Upsilon^b+\hat\Upsilon_b)\mu_b+
(\Upsilon^b+\hat\Upsilon^b)\alpha_b\sigma\,.
\end{array}$$

There is a non-degenerate skew form defined on ${\mathbb{T}}$ by
\begin{equation}\label{bigJ}\left\langle
\left[\begin{array}c\sigma\\ \mu_b\\ \rho\end{array}\right],
\left[\begin{array}c\tilde\sigma\\ \tilde\mu_c\\ \tilde\rho\end{array}\right]
\right\rangle
=\sigma\tilde\rho-J^{bc}\mu_b\tilde\mu_c-\rho\tilde\sigma
=\sigma\tilde\rho+\mu^b\tilde\mu_b-\rho\tilde\sigma,\end{equation}
(which one readily checks is preserved by~(\ref{changeofsplitting})). 

Although not yet the tractor connection, consider the connection $D_a$ on
${\mathbb{T}}$ defined by
\begin{equation}\label{fedosov_proto_tractors}
D_a\left[\begin{array}c\sigma\\ \mu_b\\ \rho\end{array}\right]=
\left[\begin{array}c\nabla_a\sigma-\mu_a\\ 
\nabla_a\mu_b-J_{ab}\rho+\Rho_{ab}\sigma-J_{ab}\alpha^c\mu_c\\ 
\nabla_a\rho-\Rho_a{}^b\mu_b
-(2\alpha^b\Rho_{ab}+\alpha^b\nabla_a\alpha_b)\sigma
\end{array}\right].\end{equation}
\begin{prop}\label{well-defined}
The connection\/ {\rm(\ref{fedosov_proto_tractors})} is
well-defined, i.e.~is independent of choice of representatives
$(J_{ab},\nabla_a)$, and preserves the skew form\/~{\rm(\ref{bigJ})}.
\end{prop}
\begin{proof} Recall that (\ref{mastereq}) can be rewritten according to 
Proposition~\ref{finite_type} as
$$\nabla_aJ^{bc}=2\alpha^{[b}\delta_a{}^{c]}.$$
We shall show in Lemma~\ref{nablaUpsilon} below that this leads to
\begin{equation}\label{another_key_identity}
\nabla_a\Upsilon^b=J^{bc}\nabla_a\Upsilon_c+
\Upsilon_a\alpha^b+\Upsilon^c\alpha_c\delta_a{}^b.\end{equation}
For convenience, let
$$T_a\equiv 2\alpha^b\Rho_{ab}+\alpha^b\nabla_a\alpha_b.$$
Now we compute
$$\begin{array}{l}
\hat D_a\left[\begin{array}c\hat\sigma\\ \hat\mu_b\\ \hat\rho
\end{array}\right]=
\left[\begin{array}{c}\hat\nabla_a\hat\sigma-\hat\mu_a\\ 
\hat\nabla_a\hat\mu_b-\hat J_{ab}\hat\rho+\hat\Rho_{ab}\hat\sigma
-\hat J_{ab}\hat\alpha^c\hat\mu_c\\ 
\hat\nabla_a\hat\rho+\hat\Rho_{ab}\hat\mu^b
-\hat T_a\hat\sigma
\end{array}\right]\\[25pt]
=\left[\begin{array}c\hat\nabla_a\sigma-(\mu_a+\Upsilon_a\sigma)\\ 
\hspace{-5pt}\begin{array}{l}\hat\nabla_a(\mu_b+\Upsilon_b\sigma)
-J_{ab}(\rho-\Upsilon^c\mu_c+\Upsilon^c\alpha_c\sigma)\\
\hspace{40pt}{}+(\Rho_{ab}-\nabla_a\Upsilon_b+\Upsilon_a\Upsilon_b)\sigma
-J_{ab}(\alpha^c+\Upsilon^c)(\mu_c+\Upsilon_c\sigma)\end{array}\hspace{-5pt}\\ 
\hspace{-15pt}\begin{array}{l}
\hat\nabla_a(\rho-\Upsilon^c\mu_c+\Upsilon^c\alpha_c\sigma)\\
\hspace{60pt}{}+(\Rho_{ab}-\nabla_a\Upsilon_b+\Upsilon_a\Upsilon_b)
(\mu^b+\Upsilon^b\sigma)-\hat T_a\sigma
\end{array}\hspace{-5pt}
\end{array}\right]\\[50pt]
=\left[\begin{array}c\nabla_a\sigma+\Upsilon_a\sigma-(\mu_a+\Upsilon_a\sigma)\\
\hspace{-5pt}\begin{array}{l}
\nabla_a(\mu_b+\Upsilon_b\sigma)-\Upsilon_b(\mu_a+\Upsilon_a\sigma)
-J_{ab}(\rho-\Upsilon^c\mu_c+\Upsilon^c\alpha_c\sigma)\\
\hspace{40pt}{}+(\Rho_{ab}-\nabla_a\Upsilon_b+\Upsilon_a\Upsilon_b)\sigma
-J_{ab}(\alpha^c+\Upsilon^c)(\mu_c+\Upsilon_c\sigma)\end{array}\hspace{-5pt}\\ 
\hspace{-15pt}\begin{array}{l}
\nabla_a(\rho-\Upsilon^c\mu_c+\Upsilon^c\alpha_c\sigma)
-\Upsilon_a(\rho-\Upsilon^c\mu_c+\Upsilon^c\alpha_c\sigma)\\
\hspace{60pt}{}+(\Rho_{ab}-\nabla_a\Upsilon_b+\Upsilon_a\Upsilon_b)
(\mu^b+\Upsilon^b\sigma)-\hat T_a\sigma
\end{array}\hspace{-5pt}
\end{array}\right],\end{array}$$
which enjoys some cancellation when expanded, yielding
$$\left[\begin{array}c\nabla_a\sigma-\mu_a\\
\nabla_a\mu_b-J_{ab}\rho+\Rho_{ab}\sigma
+\Upsilon_b\left(\nabla_a\sigma-\mu_a\right)
-J_{ab}\alpha^c\mu_c\\ 
\hspace{-5pt}\begin{array}{l}
\nabla_a\rho-(\nabla_a\Upsilon^c)(\mu_c-\alpha_c\sigma)
-\Upsilon^c\nabla_a\mu_c
+\Upsilon^c\nabla_a(\alpha_c\sigma)
-\Upsilon_a\rho\\
\hspace{5pt}{}
-\Upsilon_a\Upsilon^c\alpha_c\sigma+\Rho_{ab}\mu^b-(\nabla_a\Upsilon_b)\mu^b
+\Rho_{ab}\Upsilon^b\sigma-(\nabla_a\Upsilon_b)\Upsilon^b\sigma-\hat T_a\sigma
\end{array}\hspace{-5pt}
\end{array}\right]$$
and, if we substitute for $\nabla_a\Upsilon^c$ in accordance with 
(\ref{another_key_identity}), then a little more cancellation occurs, yielding
$$\left[\begin{array}c\nabla_a\sigma-\mu_a\\
\nabla_a\mu_b-J_{ab}\rho+\Rho_{ab}\sigma
+\Upsilon_b\left(\nabla_a\sigma-\mu_a\right)
-J_{ab}\alpha^c\mu_c\\ 
\hspace{-5pt}\begin{array}{l}
\nabla_a\rho+\Rho_{ab}\mu^b-\hat T_a\sigma+2\Upsilon^b\Rho_{ab}\sigma
-\Upsilon^b(\nabla_a\Upsilon_b)\sigma\\
\quad{}-\alpha^b(\nabla_a\Upsilon_b)\sigma
+\Upsilon^b\alpha_b\alpha_a\sigma
+\Upsilon^b(\nabla_a\alpha_b)\sigma-\Upsilon_a\Upsilon^b\alpha_b\sigma\\
\qquad{}-\Upsilon^b
(\nabla_a\mu_b-J_{ab}\rho+\Rho_{ab}\sigma-J_{ab}\alpha^c\mu_c)
+\Upsilon^b\alpha_b(\nabla_a\sigma-\mu_a)
\end{array}\hspace{-5pt}
\end{array}\right].$$
But in Lemma~\ref{magic} below we show that
\begin{equation}\label{tricky}
\begin{array}{rcl}\hat{T}_a&=&T_a+2\Upsilon^b\Rho_{ab}
-\Upsilon^b\nabla_a\Upsilon_b\\
&&\qquad{}-\alpha^b\nabla_a\Upsilon_b+\Upsilon^b\alpha_b\alpha_a
+\Upsilon^b\nabla_a\alpha_b-\Upsilon_a\Upsilon^b\alpha_b
\end{array}\end{equation}
and so this expression reduces to
$$\left[\begin{array}c\nabla_a\sigma-\mu_a\\
\nabla_a\mu_b-J_{ab}\rho+\Rho_{ab}\sigma
-J_{ab}\alpha^c\mu_c+\Upsilon_b\left(\nabla_a\sigma-\mu_a\right)\\ 
\hspace{-5pt}\begin{array}{l}
\nabla_a\rho+\Rho_{ab}\mu^b-T_a\sigma\\
\quad{}
-\Upsilon^b(\nabla_a\mu_b-J_{ab}\rho+\Rho_{ab}\sigma-J_{ab}\alpha^c\mu_c)
+\Upsilon^b\alpha_b(\nabla_a\sigma-\mu_a)
\end{array}\hspace{-5pt}
\end{array}\right],$$
which is exactly
$$\mbox{\Large$\widehat{\mbox{\normalsize$D_a
\left[\begin{array}c\sigma\\ \mu_b\\ \rho\end{array}\right]$}}$}=
\mbox{\Huge$
\widehat{\mbox{\normalsize$\left[\begin{array}c\nabla_a\sigma-\mu_a\\ 
\nabla_a\mu_b-J_{ab}\rho+\Rho_{ab}\sigma-J_{ab}\alpha^c\mu_c\\ 
\nabla_a\rho-\Rho_a{}^b\mu_b-T_a\sigma
\end{array}\right]$}}$},$$
as required.

Finally, we compute
$$\begin{array}l\left\langle D_a
\left[\begin{array}c\sigma\\ \mu_b\\ \rho\end{array}\right],
\left[\begin{array}c\tilde\sigma\\ \tilde\mu_c\\ \tilde\rho\end{array}\right]
\right\rangle+\left\langle
\left[\begin{array}c\sigma\\ \mu_b\\ \rho\end{array}\right],D_a
\left[\begin{array}c\tilde\sigma\\ \tilde\mu_c\\ \tilde\rho\end{array}\right]
\right\rangle=\\[22pt]
\quad{}\nabla_a(\sigma\tilde\rho)
-J^{bc}\nabla_a(\mu_b\tilde\mu_c)
-\alpha^b\mu_b\tilde\mu_a+\alpha^c\mu_a\tilde\mu_c
-\nabla_a(\rho\tilde\sigma)=\\[4pt]
\qquad\nabla_a(\sigma\tilde\rho)
-J^{bc}\nabla_a(\mu_b\tilde\mu_c)
-\alpha^b\delta_a{}^c\mu_b\tilde\mu_c+\alpha^c\delta_a{}^b\mu_b\tilde\mu_c
-\nabla_a(\rho\tilde\sigma)=\\[4pt]
\quad\qquad\nabla_a(\sigma\tilde\rho)
-J^{bc}\nabla_a(\mu_b\tilde\mu_c)
-(\nabla_aJ^{bc})\mu_b\tilde\mu_c-\nabla_a(\rho\tilde\sigma)=\\[4pt]
\quad\qquad\qquad
\nabla_a(\sigma\tilde\rho-J^{bc}\mu_b\tilde\mu_c-\rho\tilde\sigma)=
\nabla_a
\left\langle\left[\begin{array}c\sigma\\ \mu_b\\ \rho\end{array}\right],
\left[\begin{array}c\tilde\sigma\\ \tilde\mu_c\\ \tilde\rho\end{array}\right]
\right\rangle,
\end{array}$$
as required.
\end{proof}

\begin{lemma}\label{nablaUpsilon} The identity\/ 
{\rm(\ref{another_key_identity})} holds.
\end{lemma}
\begin{proof} We compute
$$\nabla_a\Upsilon^b=\nabla_a\left(J^{bc}\Upsilon_c\right)=
J^{bc}\nabla_a\Upsilon_c+\left(\nabla_aJ^{bc}\right)\Upsilon_c$$
and we substitute from (\ref{newmastereq}) to conclude that 
$$\nabla_a\Upsilon^b=J^{bc}\nabla_a\Upsilon_c
+\left(\alpha^b\delta_a{}^c-\alpha^c\delta_a{}^b\right)\Upsilon_c=
J^{bc}\nabla_a\Upsilon_c
+\alpha^b\Upsilon_a-\alpha^c\Upsilon_c\delta_a{}^b,$$
as required.\end{proof}
\begin{lemma}\label{magic} The identity\/ 
{\rm(\ref{tricky})} holds.
\end{lemma}
\begin{proof}
We compute
$$\begin{array}{rcl}2\hat\alpha^b\hat\Rho_{ab}
&=&2(\alpha^b+\Upsilon^b)(\Rho_{ab}-\nabla_a\Upsilon_b+\Upsilon_a\Upsilon_b)\\
&=&2\alpha^b\Rho_{ab}+2\Upsilon^b\Rho_{ab}
-2\Upsilon^b\nabla_a\Upsilon_b
-2\alpha^b\nabla_a\Upsilon_b
-2\Upsilon_a\Upsilon^b\alpha_b
\end{array}$$
and
$$\begin{array}{rcl}\hat\alpha^b\hat\nabla_a\hat\alpha_b
&=&(\alpha^b+\Upsilon^b)\nabla_a(\alpha_b+\Upsilon_b)
-(\alpha^b+\Upsilon^b)\Upsilon_b(\alpha_a+\Upsilon_a)\\
&=&\alpha^b\nabla_a\alpha_b+\Upsilon^b\nabla_a\Upsilon_b\\
&&\qquad{}+\alpha^b\nabla_a\Upsilon_b+\Upsilon^b\alpha_b\alpha_a
+\Upsilon^b\nabla_a\alpha_b+\Upsilon_a\Upsilon^b\alpha_b.
\end{array}$$
Adding these two equations gives~(\ref{tricky}), as required.
\end{proof}


\begin{prop}
The following two homomorphisms\/ 
${\mathbb{T}}\to\Lambda^1\otimes{\mathbb{T}}$ 
$$\left[\begin{array}c\sigma\\ \mu_b\\ \rho\end{array}\right]\mapsto
\left[\begin{array}c0\\ \Phi_{ab}\sigma\\ 
\Phi_{ab}\mu^b+2(\nabla^b\Phi_{ab})\sigma
\end{array}\right]
\quad\mbox{or}\quad
\left[\begin{array}c0\\ 0\\ 
(\nabla^b\Phi_{ab}+\alpha^a\Phi_{ab})\sigma
\end{array}\right]$$
are invariantly defined.
\end{prop}
\begin{proof}
Since 
$\hat\nabla_c\hat\Phi_{ab}=\hat\nabla_c\Phi_{ab}=\nabla_c\Phi_{ab}
-2\Upsilon_c\Phi_{ab}-\Upsilon_a\Phi_{cb}-\Upsilon_b\Phi_{ac}$ and 
$\hat\alpha_a=\alpha_a+\Upsilon_a$ it follows that 
$$\hat\nabla^b\hat\Phi_{ab}=\nabla^b\Phi_{ab}-\Upsilon^b\Phi_{ab}
\quad\mbox{and}\quad
\hat\alpha^a\hat\Phi_{ab}=\alpha^b\Phi_{ab}+\Upsilon^a\Phi_{ab}.$$
The required verifications are immediate.
\end{proof}
Finally, the {\em tractor connection\/} on ${\mathbb{T}}$ is defined by 
modifying $D_a$ from (\ref{fedosov_proto_tractors}) by appropriate multiples 
of these homomorphisms. The precise formula is
\begin{equation}\label{tractor_connection}
\nabla_a\!\left[\!\begin{array}c\sigma\\ \mu_b\\ \rho\end{array}\!\right]
\!\equiv\!
\left[\!\!\!\begin{array}c\nabla_a\sigma-\mu_a\\ 
\nabla_a\mu_b-J_{ab}\rho+\Rho_{ab}\sigma-\frac{3}{2n-1}\Phi_{ab}\sigma
-J_{ab}\alpha^c\mu_c\\ 
\begin{array}{l}
\nabla_a\rho+\Rho_{ab}\mu^b-\frac{3}{2n-1}\Phi_{ab}\mu^b
-\frac{1}{2n+1}(\nabla^b\Phi_{ab})\sigma\\
\enskip\quad{}-(2\alpha^b\Rho_{ab}+\alpha^b\nabla_a\alpha_b
-\frac{10n+7}{(2n+1)(2n-1)}\alpha^b\Phi_{ab})\sigma\end{array}
\end{array}\!\!\!\right]\!.\end{equation}
\begin{thm}\label{mainthm} This connection is
well-defined, i.e.~is independent of choice of representatives
$(J_{ab},\nabla_a)$. It preserves the skew form\/~{\rm(\ref{bigJ})}. Its 
curvature is given by 
$$\begin{array}{rcl}(\nabla_a\nabla_b-\nabla_b\nabla_a)\!
\left[\begin{array}c\sigma\\ \mu_c\\ \rho\end{array}\right]&\!\!=\!\!&
\left[\begin{array}c0\\
V_{abcd}\mu^d+Y_{abc}\sigma\\
Y_{abc}\mu^c
-\frac{1}{2n}(\nabla^cY_{abc}-V_{abce}\Phi^{ce})\sigma
\end{array}\right]\\[20pt]
&&{}-2J_{ab}\left[\begin{array}c\rho\\S_c\sigma -\Phi_{cd}\mu^d\\
S_c\mu^c
-\frac{1}{2n}(\Phi_{de}\Phi^{de}+\nabla^cS_c)\sigma\end{array}\right]
\end{array}$$
in Fedosov gauge.
\end{thm}
\begin{proof}
Mostly, these properties are inherited from the corresponding properties of
$D_a$ as demonstrated in Proposition~\ref{well-defined}. It only remains to
compute its curvature. According to (\ref{Rho_v_Phi}) the tractor connection in
Fedosov gauge is given by 
\begin{equation}\label{tractorconnectioninFedosovgauge}
\nabla_a\left[\begin{array}c\sigma\\ \mu_b\\
\rho\end{array}\right]= \left[\begin{array}c\nabla_a\sigma-\mu_a\\
\nabla_a\mu_b-J_{ab}\rho+\Phi_{ab}\sigma\\ 
\nabla_a\rho+\Phi_{ab}\mu^b-S_a\sigma
\end{array}\right],\end{equation}
where recall that $S_a\equiv\frac{1}{2n+1}\nabla^b\Phi_{ab}$.

We compute
$$\begin{array}{l}
\nabla_a\nabla_b\left[\begin{array}c\sigma\\ \mu_c\\ \rho\end{array}\right]
=\nabla_a\left[\begin{array}c\nabla_b\sigma-\mu_b\\ 
\nabla_b\mu_c-J_{bc}\rho+\Phi_{bc}\sigma\\ 
\nabla_b\rho+\Phi_{bc}\mu^c-S_b\sigma
\end{array}\right]\\[20pt]
\enskip{}=\left[\begin{array}c\nabla_a(\nabla_b\sigma-\mu_b)-
(\nabla_b\mu_a-J_{ba}\rho+\Phi_{ba}\sigma)\\ 
\begin{array}{r}\nabla_a(\nabla_b\mu_c-J_{bc}\rho+\Phi_{bc}\sigma)-
J_{ac}(\nabla_b\rho+\Phi_{bd}\mu^d-S_b\sigma)\qquad\\
{}+\Phi_{ac}(\nabla_b\sigma-\mu_b)\end{array}\\ 
\begin{array}{r}\nabla_a(\nabla_b\rho+\Phi_{bc}\mu^c-S_b\sigma)
-\Phi_a{}^c(\nabla_b\mu_c-J_{bc}\rho+\Phi_{bc}\sigma)\qquad\\
{}-S_a(\nabla_b\sigma-\mu_b)\end{array}
\end{array}\right].
\end{array}$$
Therefore,
$$\begin{array}{l}(\nabla_a\nabla_b-\nabla_b\nabla_a)
\left[\begin{array}c\sigma\\ \mu_c\\ \rho\end{array}\right]\\[20pt]
{}=\left[\begin{array}c-2J_{ab}\rho\\ 
\!\!\!\begin{array}{r}(\nabla_a\nabla_b-\nabla_b\nabla_a)\mu_c
-J_{ac}\Phi_{bd}\mu^d-\Phi_{ac}\mu_b
+J_{bc}\Phi_{ad}\mu^d+\Phi_{bc}\mu_a\enskip\\
{}+(\nabla_a\Phi_{bc}-\nabla_b\Phi_{ac}+J_{ac}S_b-J_{bc}S_a)\sigma
\end{array}\!\!\!\\
\begin{array}{r}
(\nabla_a\Phi_{bc}-\nabla_b\Phi_{ac}+J_{ac}S_b-J_{bc}S_a)\mu^c\qquad\\
{}-(\nabla_aS_b-\nabla_bS_a+2\Phi_a{}^c\Phi_{bc})\sigma\end{array}
\end{array}\right].
\end{array}$$
However, from (\ref{branched}) we see that
$$\begin{array}{l}(\nabla_a\nabla_b-\nabla_b\nabla_a)\mu_c=
R_{abcd}\mu^d\\
\qquad{}=V_{abcd}\mu^d
+J_{ac}\Phi_{bd}\mu^d-J_{bc}\Phi_{ad}\mu^d-\Phi_{bc}\mu_a
+\Phi_{ac}\mu_b
+2J_{ab}\Phi_{cd}\mu^d
\end{array}$$
and, if we also substitute from~(\ref{CottonYork}), then we obtain
$$\begin{array}{l}(\nabla_a\nabla_b-\nabla_b\nabla_a)
\left[\begin{array}c\sigma\\ \mu_c\\ \rho\end{array}\right]\\[20pt]
\enskip{}=\left[\begin{array}c-2J_{ab}\rho\\
V_{abcd}\mu^d+2J_{ab}\Phi_{cd}\mu^d+(Y_{abc}-2J_{ab}S_c)\sigma\\
(Y_{abc}-2J_{ab}S_c)\mu^c-(\nabla_aS_b-\nabla_bS_a+2\Phi_a{}^c\Phi_{bc})\sigma
\end{array}\right].
\end{array}$$
Lemma~\ref{fabulous_new_identity} below allows us to rewrite this expression 
as
$$\left[\!\!\begin{array}c-2J_{ab}\rho\\
V_{abcd}\mu^d+2J_{ab}\Phi_{cd}\mu^d+(Y_{abc}-2J_{ab}S_c)\sigma\\
(Y_{abc}-2J_{ab}S_c)\mu^c
-\frac{1}{2n}\big(\nabla^cY_{abc}\!-\!V_{abce}\Phi^{ce}\!-\!2J_{ab}
(\Phi_{de}\Phi^{de}\!+\!\nabla^cS_c)\big)\sigma\!\!
\end{array}\right],$$
as required.\end{proof}

\begin{lemma}\label{fabulous_new_identity}The identity
$$\nabla^cY_{abc}
=V_{abce}\Phi^{ce}
+2J_{ab}(\Phi_{de}\Phi^{de}+\nabla^cS_c)
+2n(\nabla_aS_b-\nabla_bS_a+2\Phi_a{}^c\Phi_{bc})$$
holds in Fedosov gauge.
\end{lemma}
\begin{proof}
Using (\ref{def_of_curvature}) to commute derivatives
$$\begin{array}{rcl}\nabla_d\nabla_a\Phi_{bc}
&=&(\nabla_d\nabla_a-\nabla_a\nabla_d)\Phi_{bc}+\nabla_a\nabla_d\Phi_{bc}\\
&=&R_{dabe}\Phi^e{}_c+R_{dace}\Phi_b{}^e+\nabla_a\nabla_d\Phi_{bc},
\end{array}$$
whence
$$\begin{array}{rcl}\nabla^c\nabla_a\Phi_{bc}
&=&R{}^c{}_{abe}\Phi^e{}_c+R{}^c{}_{ace}\Phi_b{}^e+\nabla_a\nabla^c\Phi_{bc}\\
&=&R{}^c{}_{abe}\Phi^e{}_c+R{}^c{}_{ace}\Phi_b{}^e+(2n+1)\nabla_aS_b.
\end{array}$$
Substituting from (\ref{branched}) gives
$$\nabla^c\nabla_a\Phi_{bc}=V_{acbe}\Phi^{ce}+2n\Phi_a{}^c\Phi_{bc}
+J_{ab}\Phi_{de}\Phi^{de}+(2n+1)\nabla_aS_b$$
and, similarly,
$$\nabla^c\nabla_b\Phi_{ac}=V_{bcae}\Phi^{ce}-2n\Phi_a{}^c\Phi_{bc}
-J_{ab}\Phi_{de}\Phi^{de}+(2n+1)\nabla_bS_a.$$
Noting that $V_{acbe}-V_{bcae}=V_{abce}$, we may subtract these two equations 
to obtain
$$\begin{array}{l}\nabla^c\nabla_a\Phi_{bc}-\nabla^c\nabla_b\Phi_{ac}\\
\quad{}=V_{abce}\Phi^{ce}+4n\Phi_a{}^c\Phi_{bc}
+2J_{ab}\Phi_{de}\Phi^{de}+(2n+1)(\nabla_aS_b-\nabla_bS_a).
\end{array}$$
Therefore, from the formula~(\ref{CottonYork}) for $Y_{abc}$, we conclude that
$$\begin{array}{rcl}\nabla^cY_{abc}
&=&\nabla^c\nabla_a\Phi_{bc}-\nabla^c\nabla_b\Phi_{ac}
-(\nabla_aS_b-\nabla_bS_a)+2J_{ab}\nabla^cS_c\\
&=&V_{abce}\Phi^{ce}+4n\Phi_a{}^c\Phi_{bc}
+2J_{ab}\Phi_{de}\Phi^{de}\\
&&\qquad{}+2n(\nabla_aS_b-\nabla_bS_a)
+2J_{ab}\nabla^cS_c,
\end{array}$$
as required.\end{proof}

Theorem~\ref{mainthm} has the following immediate consequence.
\begin{cor}\label{maincorollary} 
The curvature of the tractor connection has the form
\begin{equation}\label{Einstein}
(\nabla_a\nabla_b-\nabla_b\nabla_a)\Sigma=2J_{ab}\Theta\Sigma\end{equation}
for some endomorphism $\Theta$ of\/ ${\mathbb{T}}$ if and only if\/ 
$V_{abcd}\equiv 0$. 
\end{cor}
\begin{proof}
Notice that the curvature in the statement of Theorem~\ref{mainthm} is split 
already into its irreducible components according to
\begin{equation}\label{curvature_split}\Lambda^2\otimes\End({\mathbb{T}})=
\big(\Lambda_\perp^2\otimes\End({\mathbb{T}})\big)
\oplus\End({\mathbb{T}}),\end{equation}
where $\Lambda_\perp^2$ denotes the $2$-forms that are trace-free with respect 
to~$J_{ab}$. That the curvature has the form (\ref{Einstein}) is precisely 
that the component in $\Lambda_\perp^2\otimes\End({\mathbb{T}})$ vanish, i.e. 
that 
$$\left[\begin{array}c\sigma\\ \mu_b\\ \rho\end{array}\right]\longmapsto
\left[\begin{array}c0\\
V_{abcd}\mu^d+Y_{abc}\sigma\\
Y_{abc}\mu^c
-\frac{1}{2n}(\nabla^cY_{abc}-V_{abce}\Phi^{ce})\sigma
\end{array}\right]$$
vanish identically. Clearly this implies that $V_{abcd}\equiv 0$ but then the
contracted Bianchi identity (\ref{contractedBianchi}) implies that
$Y_{abc}\equiv 0$.
\end{proof}
\begin{cor} The endomorphism\/ $\Theta:{\mathbb{T}}\to{\mathbb{T}}$ defined 
in Fedosov gauge by
$$\left[\begin{array}c\sigma\\ \mu_b\\ \rho\end{array}\right]\longmapsto
\left[\begin{array}c\rho\\S_c\sigma -\Phi_{cd}\mu^d\\
S_c\mu^c
-\frac{1}{2n}(\Phi_{de}\Phi^{de}+\nabla^cS_c)\sigma\end{array}\right]$$
respects the skew form\/~{\rm(\ref{bigJ})}.
\end{cor}
\begin{proof} One may check by direct calculation that
$\langle\Theta\Sigma,\tilde\Sigma\rangle
+\langle\Sigma,\Theta\tilde\Sigma\rangle=0$. Alternatively, whether or not
$V_{abcd}=0$, the endomorphism $\Theta$ can be viewed via
(\ref{curvature_split}) and Theorem~\ref{mainthm} as an irreducible component
of the tractor curvature. As the tractor connection respects~(\ref{bigJ}), so
does its curvature and any irreducible component thereof.
\end{proof}

We may further pursue the consequences of $V_{abcd}=0$ as follows. 

\begin{lemma}\label{gradTheta}
When \eqref{Einstein} holds and the homomorphism $\Theta$ is 
written in Fedosov gauge, then $\nabla_a\Theta=0$.
\end{lemma}
\begin{proof} When (\ref{Einstein}) holds, the Bianchi identity for the
connection $\nabla_a$ on ${\mathbb{T}}$ implies that
$\nabla_{[a}(J_{bc]}\Theta)=0$. In Fedosov gauge, this may be rewritten as
$J_{[bc}\nabla_{a]}\Theta=0$. Non-degeneracy of $J_{bc}$ implies that 
$\nabla_a\Theta=0$.
\end{proof}
{From} Theorem~\ref{mainthm}, when (\ref{Einstein}) holds the homomorphism
$\Theta$ is given in Fedosov gauge by 
$$\Theta\!\left[\begin{array}c\sigma\\
\mu_b\\ \rho\end{array}\right]= \left[\begin{array}c
-\rho\\\Phi_{bc}\mu^c-S_b\sigma\\
X\sigma-S_c\mu^c\end{array}\right],\enskip
\mbox{where }X\equiv\textstyle\frac{1}{2n}(\Phi_{de}\Phi^{de}+\nabla^cS_c).$$
But, by using the invariant symplectic form (\ref{bigJ}) on~${\mathbb{T}}$, we
can equally well view $\Theta$ as a section of
{\large$\otimes$}$^2{\mathbb{T}}$. Specifically, 
\begin{equation}\label{Theta}
\Theta=\left[\begin{array}{ccc} -1&0&0\\0&-\Phi_{bc}&S_b\\
0&S_c&-X\end{array}\right].\end{equation}
Note that $\Theta$ is symmetric (as must be the case since $\nabla_a$ preserves
the symplectic form (\ref{bigJ}) on~${\mathbb{T}}$). 
\begin{thm}
If\/ $V_{abcd}=0$, then
\begin{equation}\label{mobility}(\nabla_a\Phi^{bc})_\circ=0\end{equation}
in Fedosov gauge, where $(\enskip)_\circ$ means to take the trace-free part.
\end{thm}
\begin{proof}
{From} (\ref{tractorconnectioninFedosovgauge}) and (\ref{Theta}) we compute
$$\nabla_a\Theta=\left[\begin{array}{ccc} 0&0&0\\
0&-\nabla_a\Phi_{bc}-J_{ab}S_c-J_{ac}S_b&
\nabla_aS_b+J_{ab}X-\Phi_{ac}\Phi_{b}{}^c\\
0&\nabla_aS_c-\Phi_{ab}\Phi^b{}_c+J_{ac}X&
-\nabla_aX+\Phi_{ab}S^b+\Phi_{ac}S^c\end{array}\right].$$
{From} Lemma~\ref{gradTheta} we conclude that 
\begin{equation}\label{gradPhi}\nabla_a\Phi_{bc}+J_{ab}S_c+J_{ac}S_b=0
\end{equation}
and raising indices with $J^{ab}$ gives
$$\nabla_a\Phi^{bc}+\delta_a{}^bS^c+\delta_a{}^cS^b=0,$$
as required.
\end{proof}
Several remarks are in order. Firstly, notice that (\ref{gradPhi}) is only an
extra condition on $\nabla_{(a}\Phi_{b)c}$ since
$\nabla_{[a}\Phi_{b]c}+J_{ab}S_c-J_{c[a}S_{b]}=\frac{1}{2}Y_{abc}$ in
accordance with~(\ref{CottonYork}). Secondly, the partial differential
equations~(\ref{mobility}) are the well-known {\em mobility\/}
equations~\cite{M} of projective differential geometry whose non-degenerate
solutions $\Phi^{ab}$ are in one-to-one correspondence with 
(pseudo-)Riemannian metrics having connection in the projective class
$[\nabla_a]$ of~$\nabla_a$. Thirdly, the other components of $\nabla_a\Theta$
apparently give rise to a whole system of equations,
$$\begin{array}{rcl}
\nabla_a\Phi^{bc}+\delta_a{}^bS^c+\delta_a{}^cS^b&=&0\\
\nabla_aS^b+\delta_a{}^bX-\Phi_{ac}\Phi^{bc}&=&0\\
\nabla_aX-2\Phi_{ab}S^b&=&0
\end{array}$$
but, in fact, this is exactly the prolongation of the (\ref{mobility}) as
derived in~\cite{EM}. Therefore, the vanishing of $\nabla_a\Theta$ is precisely
equivalent to the mobility equations (\ref{mobility}) on~$\Phi^{ab}$. As
described in~\cite{EM}, for (\ref{mobility}) to admit any non-zero solutions
imposes further non-trivial conditions on the projective
structure~$[\nabla_a]$. If $\Phi_{ab}\equiv0$, however, then the connection
$\nabla_a$ from the Fedosov gauge is flat, as can be seen 
from~(\ref{branched}). Finally, notice that the partial differential equations
(\ref{mobility}) are actually much stronger than the mobility equations alone
because $\Phi^{ab}$ is actually part of the curvature of~$\nabla_a$.

\subsection{Examples}\label{examples}
In view of the strength of equations (\ref{mobility}) it is not easy to provide
any non-trivial examples of a conformally Fedosov structure with $V_{abcd}=0$.
Complex projective space ${\mathbb{CP}}_n$ with its usual projective structure
and symplectic form certainly provides the best example. In this case, recall
that
$$V_{abcd}=0\quad\mbox{and}\quad\Phi_{ab}=g_{ab}$$
so ${\mathbb{CP}}_n$ is not projectively flat. In Fedosov gauge $S_a=0$ and 
$$\Theta=\left[\begin{array}{ccc} -1&0&0\\0&-g_{bc}&0\\
0&0&-1\end{array}\right]$$
as a section of {\large$\odot$}$^2{\mathbb{T}}$.

Another example may be based on $S^1\times S^{2n-1}$ with its
conformally symplectic structure induced by the dilation invariant 
$$J_{ab}\equiv\left(1/{\|x\|}\right)^2\omega_{ab}$$
on~${\mathbb{R}}^{2n}\setminus\{0\}$, where $\omega_{ab}$ is the standard
symplectic form
$$\omega_{ab}\,dx^a\,dx^b=dx^1\wedge dx^2+dx^3\wedge dx^4+\cdots\,.$$
The flat connection $\partial_a$ on ${\mathbb{R}}^{2n}\setminus\{0\}$ is also
dilation invariant whence the pair $(J_{ab},\partial_a)$ defines a dilation
invariant conformally Fedosov structure on~${\mathbb{R}}^{2n}\setminus\{0\}$.
Indeed, 
$$\partial_aJ_{bc}=
2\big(\partial_a\log(1/\|x\|)\big)J_{bc}=-2(x_a/\|x\|^2)J_{bc}$$
so (\ref{eqns}) holds with $\alpha_a=-x_a/\|x\|^2$ and $\beta_a=-2x_a/\|x\|^2$. 
The proof of Proposition~\ref{making_the_mastereq} shows that we should take
$$\nabla_a\phi_b=\partial_a\phi_b+x_a\phi_b/\|x\|^2+x_b\phi_a/\|x\|^2$$
as the unique projectively flat connection so that (\ref{mastereq}) holds. 
Notice that, although 
$$V_{abcd}=0\quad\mbox{and}\quad\Phi_{ab}=0$$
in this case, the curvature of the corresponding tractor connection is not
flat. Indeed, we have 
$(\nabla_a\nabla_b-\nabla_b\nabla_a)\Sigma=2J_{ab}\Theta\Sigma$ where, as a 
section of {\large$\odot$}$^2{\mathbb{T}}$,
$$\Theta=\left[\begin{array}{ccc} -1&x_c/\|x\|^2&-1/\|x\|^2\\[3pt]
x_b/\|x\|^2&-x_bx_c/\|x\|^4&x_b/\|x\|^4\\[3pt]
-1/\|x\|^2&x_c/\|x\|^4&-1/\|x\|^4\end{array}\right],$$
which has rank 1.

\section{Calculus on conformally symplectic manifolds}
For most of this section, we shall work on a conformally symplectic manifold
$(M,[J])$ by choosing a non-degenerate $2$-form $J_{ab}$ from the conformal
class~$[J]$. Only after we have established Theorem~\ref{coupledRStheorem}
shall we restore invariance by considering how matters change under conformal
rescaling $J_{ab}\mapsto\hat{J}_{ab}=\Omega^2J_{ab}$.

It has recently been noticed~\cite{S,TY} that on a symplectic manifold, there 
is a natural alternative to the de~Rham complex, which begins
$$0\to\Lambda^0\xrightarrow{\,d\,}\Lambda^1\to\Lambda_\perp^2\to\cdots$$
where $\Lambda_\perp^k$ denotes the bundle of $k$-forms that are trace-free
with respect to~$J$. In~\cite{E2}, which was written as a precursor to the
current article, this construction was generalised to conformally symplectic
manifolds. The result is an elliptic complex
\begin{equation}\label{conformalRS}\begin{array}{rcccccccccccc}
0&\to&\Lambda^0&\xrightarrow{d-2\alpha}&\Lambda^1
&\to&\Lambda_\perp^2
&\to&\Lambda_\perp^3
&\to&\cdots
&\to&\Lambda_\perp^{n}\\[2pt]
&&&&&&&&&&&&\big\downarrow\\
0&\leftarrow&\Lambda^0&\longleftarrow&\Lambda^1
&\leftarrow&\Lambda_\perp^2
&\leftarrow&\Lambda_\perp^3
&\leftarrow&\cdots
&\leftarrow&\Lambda_\perp^{n}
\end{array}\end{equation}
where all operators are first order except for the middle operator, which is
second order. In \S\ref{tractors}, we constructed, for any conformally Fedosov
manifold, its natural {\em tractor bundle\/} ${\mathbb{T}}$ equipped with a
connection $\nabla_a$ having curvature of the form (\ref{Einstein}) for some
endomorphism $\Theta$ of ${\mathbb{T}}$ if and only if $V_{abcd}=0$ (and in
\S\ref{examples} exhibited some conformally Fedosov structures with
$V_{abcd}=0$). In this section we explore the consequences of~(\ref{Einstein}).
In the first instance, the particular bundle ${\mathbb{T}}$ need not concern us
and, instead, we develop a theory of {\em symplectically flat connections\/} on
a conformally symplectic manifolds. Specifically, we shall say that a
connection $\nabla_a$ on a given smooth vector bundle $E$ over a conformally
symplectic manifold $(M,[J])$ is {\em symplectically flat\/} if and only if
\begin{equation}\label{symplectically_flat}
(\nabla_a\nabla_b-\nabla_b\nabla_a)\sigma=2J_{ab}\Theta\sigma
\end{equation} 
for some endomorphism $\Theta$ of~$E$ (as usual, one chooses an arbitrary
torsion-free connection on $\Lambda^1$ to define the left hand side of
(\ref{symplectically_flat}), which then does not depend on this choice). 
Evidently, if $J_{ab}$ is replaced by $\hat{J}_{ab}=\Omega^2J_{ab}$, then 
(\ref{symplectically_flat}) persists with $\Theta$ replaced by 
$\hat\Theta=\Omega^{-2}\Theta$. For symplectically flat connections, our
aim is to construct a version of (\ref{conformalRS}) coupled to~$E$. It is
clear how this complex should start, namely
$$E\xrightarrow{\,\nabla-2\alpha\otimes{\mathrm{Id}}\,}\Lambda^1\otimes E
\longrightarrow\Lambda_\perp^2\otimes E,$$
where $\Gamma(\Lambda^1\otimes E)\ni\phi_a\mapsto
\nabla_{[a}\phi_{b]}-2\alpha_{[a}\phi_{b]}\bmod J_{ab}$ since then
$$\phi_a=\nabla_a\sigma-2\alpha_a\sigma\Rightarrow
\nabla_{[a}\phi_{b]}-2\alpha_{[a}\phi_{b]}
=\nabla_{[a}\nabla_{b]}\sigma=J_{ab}\Theta\sigma=0\bmod J_{ab},$$ 
as required. We shall now construct the coupled version of (\ref{conformalRS}) 
from scratch, including the construction of (\ref{conformalRS}) itself.

The operator
$$D_a=\nabla_a-2\alpha_a:E\to\Lambda^1\otimes E$$
is a connection whose curvature we have, in effect, just computed:
\begin{equation}\label{symplectically_flat_too}(D_aD_b-D_bD_a)\sigma
=(\nabla_a\nabla_b-\nabla_b\nabla_a)\sigma=2J_{ab}\Theta\sigma,\end{equation}
where, as usual, an arbitrary but irrelevant torsion-free connection has been
chosen on $\Lambda^1$ to define $D:\Lambda^1\otimes
E\to\Lambda^1\otimes\Lambda^1\otimes E$. \begin{lemma} The endomorphism
$\Theta:E\to E$ has constant rank.
\end{lemma}
\begin{proof} Locally, we can work in Fedosov gauge to conclude, arguing as in
the proof of Lemma~\ref{gradTheta}, that $\nabla_a\Theta=0$. Having done this,
along any smooth curve in $M$ we can choose, for the 
connection~$\nabla_a$, a covariant constant frame for~$E$.
In this frame, the endomorphism $\Theta$ is then represented by a constant
matrix whose rank is therefore constant.
\end{proof}
\begin{lemma}\label{kerTheta}
The connection $D_a$ on $E$ induces a flat connection on the subbundle
$\ker\Theta\subseteq E$.
\end{lemma}
\begin{proof} If $\sigma\in\Gamma(\ker\Theta)$, then
$$\Theta D_a\sigma=D_a(\Theta\sigma)=0$$
so $D_a\sigma\in\Gamma(\Lambda^1\otimes\ker\Theta)$. That 
$D_a:\ker\Theta\to\Lambda^1\otimes\ker\Theta$ is a connection follows 
immediately and (\ref{symplectically_flat_too}) shows that it is flat.
\end{proof}
\begin{lemma}\label{cokerTheta}
The connection $D_a$ on $E$ induces a flat connection on the bundle
$\coker\Theta\equiv E/\im\Theta$.
\end{lemma}
\begin{proof}Suppose $\eta\in\Gamma(E)$ is of the form $\eta=\Theta\sigma$ for 
some $\sigma\in\Gamma(E)$. Then
$$D_a\eta=D_a(\Theta\sigma)
=\Theta D_a\sigma\in\Gamma(\Lambda^1\otimes\im\Theta).$$
Therefore $D_a:E\to\Lambda^1\otimes E$ descends to 
$D_a:\coker\Theta\to\Lambda^1\otimes\coker\Theta$, for which the Leibnitz rule 
is readily verified. Again (\ref{symplectically_flat_too}) implies that this 
connection is flat.  
\end{proof}
Let us write \underbar{$\ker\Theta$}, respectively \underbar{$\coker\Theta$},
for the sheaf of germs of covariant constant sections of $\ker\Theta$, 
respectively~$\coker\Theta$. According to Lemmata~\ref{kerTheta} 
and~\ref{cokerTheta}, these are locally constant sheaves. 
\begin{lemma}\label{mainlemma}
There is a natural elliptic complex:
$$\begin{array}{cccccccccc}E
&\stackrel{D}{\longrightarrow}&\Lambda^1\otimes E
&\stackrel{D}{\longrightarrow}&\Lambda^2\otimes E
&\stackrel{D}{\longrightarrow}&\Lambda^3\otimes E
&\stackrel{D}{\longrightarrow}&\Lambda^4\otimes E\\
&\begin{picture}(0,0)(0,-3)
\put(-9,6){\vector(3,-2){18}}\end{picture}
&\oplus
&\begin{picture}(0,0)(0,-3)
\put(-9,-6){\vector(3,2){18}}
\put(-9,6){\vector(3,-2){18}}\end{picture}
&\oplus
&\begin{picture}(0,0)(0,-3)
\put(-9,-6){\vector(3,2){18}}
\put(-9,6){\vector(3,-2){18}}\end{picture}
&\oplus
&\begin{picture}(0,0)(0,-3)
\put(-9,-6){\vector(3,2){18}}
\put(-9,6){\vector(3,-2){18}}\end{picture}
&\oplus&\cdots,
\\
&&E&\longrightarrow&\Lambda^1\otimes E
&\longrightarrow&\Lambda^2\otimes E
&\longrightarrow&\Lambda^3\otimes E
\end{array}$$
where the differentials are given by 
$$\sigma\!\mapsto\!\left[\!\begin{array}{c}D\sigma\\ 
\Theta\sigma\end{array}\!\right]
\quad
\left[\!\begin{array}{c}\phi\\ \eta\end{array}\!\right]
\!\mapsto\!\left[\!\begin{array}{c}D\phi-J\otimes\eta\\ 
D\eta-\Theta\phi\end{array}\!\right]
\quad\left[\!\begin{array}{c}\omega\\ \psi\end{array}\!\right]
\!\mapsto\!\left[\!\begin{array}{c}D\omega+J\wedge\psi\\ 
D\psi+\Theta\omega\end{array}\!\right]\enskip\cdots.$$
It is locally exact save for the zeroth and first cohomologies, which may be
identified with \underbar{$\ker\Theta$} and \underbar{$\coker\Theta$}, 
respectively.
\end{lemma}
\begin{proof}
That this is an elliptic complex is easily verified. Since $D\sigma=0$ implies 
$\Theta\sigma=0$ it is also evident that 
$$\ker\left(\sigma\!\mapsto\!\left[\!\begin{array}{c}D\sigma\\ 
\Theta\sigma\end{array}\!\right]\right)={\underbar{$\ker\Theta$}}.$$
To investigate the higher cohomology, locally let us choose a smooth $1$-form
$\tau$ such that $D\tau=d\tau-2\alpha\wedge\tau=J$, which is possible since
(\ref{def_alpha}) holds. This choice allows us to define a new connection 
$$\tilde D\equiv D-\tau\otimes\Theta:E\to\Lambda^1\otimes E$$
on $E$ for which
$$(\tilde D_a\tilde D_b-\tilde D_b\tilde D_a)\sigma
=(D_aD_b-D_bD_a)\sigma-2(D_{[a}\tau_{b]})\Theta\sigma=0.$$
In other words $\tilde D:E\to\Lambda^1\otimes E$ is a flat connection. 
This allows us to show that the cohomology at any level $p\geq 1$ can be
represented by elements of the form
\begin{equation}\label{normal_form}
\left[\!\begin{array}{c}(-1)^{p+1}\tau\wedge\eta\\[5pt] 
\eta\end{array}\!\right]
\in\;\Gamma\left(\!\!\begin{array}{c}\Lambda^p\otimes E\\ 
\oplus\\
\Lambda^{p-1}\otimes E\end{array}\!\!\right).\end{equation}
Specifically, if $p$ is odd, then we are given
$$\left[\!\begin{array}{c}\phi \\[5pt] \eta\end{array}\!\right]
\in\;\Gamma\left(\!\!\begin{array}{c}\Lambda^p\otimes E\\ 
\oplus\\
\Lambda^{p-1}\otimes E\end{array}\!\!\right)\enskip\mbox{such that}\enskip
\left\{\!\!\begin{array}{c}D\phi=J\wedge\eta\\[5pt]
D\eta=\Theta\phi\end{array}\right.$$
and it follows that $\tilde D(\phi-\tau\wedge\eta)
=D\phi-\tau\Theta\phi-(D\tau)\wedge\eta+\tau\wedge D\eta=0$.
As $\tilde D$ is flat, locally we can find 
$\sigma\in\Gamma(\Lambda^{p-1}\otimes E)$ such that 
$$\phi-\tau\wedge\eta=\tilde D\sigma=D\sigma-\tau\wedge\Theta\sigma.$$
If we now replace 
$$\left[\!\begin{array}{c}\phi\\ \eta\end{array}\!\right]
\enskip\mbox{by}\enskip
\left[\!\begin{array}{c}\phi\\ \eta\end{array}\!\right]-
\left[\!\begin{array}{c}D\sigma\\ \Theta\sigma\end{array}\!\right]
=\left[\!\begin{array}{c}\phi-D\sigma\\ \eta-\Theta\sigma
\end{array}\!\right]
=\left[\!\begin{array}{c}\tau\wedge(\eta-\Theta\sigma)\\ \eta-\Theta\sigma
\end{array}\!\right],$$
as we may, then we have obtained a representative of the required form. The 
case of $p$ even is essentially the same save for a few sign changes.
Thus, we are reduced to computing the local cohomology for elements of the 
form~(\ref{normal_form}). The remaining constraint on such elements is that 
$$D\eta=\tau\wedge\Theta\eta\quad\mbox{for}\quad
\eta\in\Gamma(\Lambda^{p-1}\otimes E)$$
but this is exactly that $\tilde D\eta=0$. For $p\geq 2$ it follows that
locally we can find $\xi\in\Gamma(\Lambda^{p-2}\otimes E)$ such that 
$\tilde D\xi=\eta$. When $p=2$, for example, we find that
$$\left[\!\begin{array}{c}\tau\otimes\xi\\ \xi\end{array}\!\right]
\!\mapsto\!\left[\!\begin{array}{c}D(\tau\otimes\xi)-J\otimes\xi\\ 
D\xi-\tau\otimes\Theta\xi\end{array}\!\right]
=\left[\!\begin{array}{c}-\tau\wedge D\xi\\ 
\tilde D\xi\end{array}\!\right]
=\left[\!\begin{array}{c}-\tau\wedge\tilde D\xi\\ 
\tilde D\xi\end{array}\!\right]$$
and so the second cohomology vanishes. The case $p\geq 3$ is similar. When 
$p=1$, the normal form (\ref{normal_form}) identifies the local cohomology as 
$$\frac{\{\eta\in\Gamma(E)\mid\tilde D\eta=0\}}
{\Theta\{\sigma\in\Gamma(E)\mid\tilde D\sigma=0\}},$$
noting that $\tilde D\Theta=D\Theta-\tau\otimes[\Theta,\Theta]=0$ so that this
quotient is well-defined. Since $\tilde D$ is flat, this is a quotient of
finite-dimensional vector spaces, which at each point may be identified with
$\coker\Theta$. Evidently, this local construction, depending on choice
of~$\tau$, agrees with the global construction of \underbar{$\coker\Theta$} 
provided by Lemma~\ref{cokerTheta}.\end{proof}

\begin{thm}[The coupled Rumin--Seshadri complex]\label{coupledRStheorem} 
Suppose $(M,[J])$ is a conformally symplectic manifold and $\nabla_a$ is a
symplectically flat connection on a vector bundle $E$ over~$M$. Choose 
$J_{ab}\in[J]$ and define
$\Theta:E\to E$ by means of~\eqref{symplectically_flat}. Then there is a
natural elliptic complex
\begin{equation}\label{coupledRS}\begin{array}{rcccccccccc}
0&\to&E&\to&\Lambda^1\otimes E
&\to&\Lambda_\perp^2\otimes E
&\to&\cdots
&\to&\Lambda_\perp^{n}\otimes E\\[2pt]
&&&&&&&&&&\big\downarrow\\
0&\leftarrow&E&\leftarrow&\Lambda^1\otimes E
&\leftarrow&\Lambda_\perp^2\otimes E
&\leftarrow&\cdots
&\leftarrow&\Lambda_\perp^{n}\otimes E
\end{array}\end{equation}
where all operators are first order save for the middle operator, which is
second order. This differential complex is locally exact save for its zeroth
and first cohomologies, which may be identified with \underbar{$\ker\Theta$}
and \underbar{$\coker\Theta$}, respectively.
\end{thm}
\begin{proof}
Rearranging the complex from Lemma~\ref{mainlemma} as 
$$\begin{array}{ccccccccccc}
E&\to&\Lambda^1\otimes E&\to&\Lambda^2\otimes E
&\to&\Lambda^3\otimes E&\to&\Lambda^4\otimes E&\to&\cdots\\
&\begin{picture}(0,0)\put(-5,13){\vector(4,-1){80}}\end{picture}&
&\begin{picture}(0,0)\put(-5,13){\line(4,-1){30}}
\put(43,1){\vector(4,-1){30}}\end{picture}&\uparrow
&\begin{picture}(0,0)\put(-5,13){\line(4,-1){30}}
\put(43,1){\vector(4,-1){30}}\end{picture}&\uparrow
&\begin{picture}(0,0)\put(-5,13){\line(4,-1){30}}
\put(43,1){\vector(4,-1){30}}\end{picture}&\uparrow
&\begin{picture}(0,0)\put(-5,13){\line(4,-1){30}}\end{picture}\\
&&&&E&\to&\Lambda^1\otimes E&\to&\Lambda^2\otimes E&\to&\cdots
\end{array}$$
one sees a filtered complex, the spectral sequence of which has as its
$E_1$-level
$$\begin{picture}(352,40)
\put(0,0){\vector(1,0){352}}
\put(0,0){\vector(0,1){40}}
\put(176,20){\makebox(0,0){\footnotesize
$\begin{array}{rcl}
E\!\to\!\Lambda^1\!\!\otimes\!E
\!\to\!\Lambda_\perp^2\!\!\otimes\!E
\!\to\!\cdots\!\to\!\Lambda_\perp^n\!\!\otimes\!E
&0&\\[6pt]
&0&\Lambda_\perp^n\!\!\otimes\!E
\!\to\!\cdots\!\to\!\Lambda_\perp^2\!\!\otimes\!E
\!\to\!\Lambda^1\!\!\otimes\!E\!\to\!E.
\end{array}$}}
\end{picture}$$
Passing to the $E_2$-level constructs the complex (\ref{coupledRS}) and 
Lemma~\ref{mainlemma} gives its cohomology.
\end{proof}

As advised at the beginning of this section, we shall now consider the effect
of replacing $J_{ab}$ by $\hat{J}_{ab}=\Omega^2J_{ab}$. Recall, that
(\ref{def_alpha}) is maintained by replacing $\alpha_a$ with
$\hat{\alpha}_a=\alpha_a+\Upsilon_a$ where $\Upsilon_a=\nabla_a\log\Omega$. It 
follows that $d-2\alpha:\Lambda^0\to\Lambda^1$ for any choice of 
$J_{ab}\in[J]$ is better regarded as an invariantly defined connection
\begin{equation}\label{D}D_a:\Lambda^0[2]\to\Lambda^1[2]\end{equation}
on the bundle $\Lambda^0[2]$ of conformal densities of weight~$2$ 
(characterised by the requirement that $D_{[a}J_{bc]}=0$). The connection 
(\ref{D}) induces connections on all the density bundles $\Lambda^0[w]$ and 
there are conformally invariant versions of (\ref{coupledRS})
\begin{equation}\label{coupledweightedRS}\begin{array}{ccccccc}
\cdots&\to&\Lambda_\perp^p\otimes E[w+2]
&\to&\cdots
&\to&\Lambda_\perp^{n}\otimes E[w+2]\\[2pt]
&&&&&&\big\downarrow\\
\cdots&\leftarrow&\Lambda_\perp^p\otimes E[w+2p-2n]
&\leftarrow&\cdots
&\leftarrow&\Lambda_\perp^{n}\otimes E[w]
\end{array}\end{equation}
for all~$w$. The formal adjoint of this complex is another of the same form but
with $w$ replaced by $-w-2$. Notice that some statements are more naturally 
made by using these invariant connections on densities. For example, the 
endomorphism $\Theta$ from (\ref{symplectically_flat}) should be regarded as 
having conformal weight $-2$ and then $D_a\Theta=0$ (replacing 
Lemma~\ref{gradTheta}). 

\section{Bernstein--Gelfand--Gelfand complexes}
Our aim in this section is typified by the following example. Let us suppose 
that $(M,[J,\nabla])$ is a four-dimensional conformally Fedosov manifold and that 
the invariant curvature $V_{abcd}$ vanishes. According to 
Corollary~\ref{maincorollary} and Theorem~\ref{coupledRStheorem} there is a 
differential complex of the form
$$0\to{\mathbb{T}}\to\Lambda^1\otimes{\mathbb{T}}\to
\Lambda_\perp^2\otimes{\mathbb{T}}\to\Lambda_\perp^2\otimes{\mathbb{T}}[-2]
\to\Lambda^1\otimes{\mathbb{T}}[-4]\to{\mathbb{T}}[-6]\to0$$
(by taking $w=-2$ in~(\ref{coupledweightedRS})), whose local cohomology we
know. To proceed, we should extend the Dynkin diagram notation for irreducible
representations of ${\mathrm{Sp}}(4,{\mathbb{R}})$ to include a conformally
symplectic weight. For reasons indicated in~\cite{BE}, it is convenient to
write
$$\begin{picture}(32,5)
\put(4,1.5){\line(1,0){12}}
\put(4,1.2){\makebox(0,0){$\times$}}
\put(16,1.2){\makebox(0,0){$\bullet$}}
\put(28,1.2){\makebox(0,0){$\bullet$}}
\put(16,0.5){\line(1,0){12}}
\put(16,2.5){\line(1,0){12}}
\put(22,1.5){\makebox(0,0){$\langle$}}
\put(4,8){\makebox(0,0){$\scriptstyle a$}}
\put(16,9){\makebox(0,0){$\scriptstyle b$}}
\put(28,8){\makebox(0,0){$\scriptstyle c$}}
\end{picture}\quad\mbox{for}\quad
\begin{picture}(20,5)
\put(4,1.2){\makebox(0,0){$\bullet$}}
\put(16,1.2){\makebox(0,0){$\bullet$}}
\put(4,0.5){\line(1,0){12}}
\put(4,2.5){\line(1,0){12}}
\put(10,1.5){\makebox(0,0){$\langle$}}
\put(4,9){\makebox(0,0){$\scriptstyle b$}}
\put(16,8){\makebox(0,0){$\scriptstyle c$}}
\end{picture}(TM)\otimes\Lambda^0[a-c]$$
so that, for example,
$$TM=\begin{picture}(32,5)
\put(4,1.5){\line(1,0){12}}
\put(4,1.2){\makebox(0,0){$\times$}}
\put(16,1.2){\makebox(0,0){$\bullet$}}
\put(28,1.2){\makebox(0,0){$\bullet$}}
\put(16,0.5){\line(1,0){12}}
\put(16,2.5){\line(1,0){12}}
\put(22,1.5){\makebox(0,0){$\langle$}}
\put(4,8){\makebox(0,0){$\scriptstyle 0$}}
\put(16,8){\makebox(0,0){$\scriptstyle 1$}}
\put(28,8){\makebox(0,0){$\scriptstyle 0$}}
\end{picture}\qquad\Lambda^1=\begin{picture}(32,5)
\put(4,1.5){\line(1,0){12}}
\put(4,1.2){\makebox(0,0){$\times$}}
\put(16,1.2){\makebox(0,0){$\bullet$}}
\put(28,1.2){\makebox(0,0){$\bullet$}}
\put(16,0.5){\line(1,0){12}}
\put(16,2.5){\line(1,0){12}}
\put(22,1.5){\makebox(0,0){$\langle$}}
\put(4,8){\makebox(0,0){$\scriptstyle -2$}}
\put(16,8){\makebox(0,0){$\scriptstyle 1$}}
\put(28,8){\makebox(0,0){$\scriptstyle 0$}}
\end{picture}\qquad\Lambda_\perp^2=\begin{picture}(32,5)
\put(4,1.5){\line(1,0){12}}
\put(4,1.2){\makebox(0,0){$\times$}}
\put(16,1.2){\makebox(0,0){$\bullet$}}
\put(28,1.2){\makebox(0,0){$\bullet$}}
\put(16,0.5){\line(1,0){12}}
\put(16,2.5){\line(1,0){12}}
\put(22,1.5){\makebox(0,0){$\langle$}}
\put(4,8){\makebox(0,0){$\scriptstyle -3$}}
\put(16,8){\makebox(0,0){$\scriptstyle 0$}}
\put(28,8){\makebox(0,0){$\scriptstyle 1$}}
\end{picture}.$$
In particular,
$${\mathbb{T}}=
\begin{picture}(32,5)
\put(4,1.5){\line(1,0){12}}
\put(4,1.2){\makebox(0,0){$\times$}}
\put(16,1.2){\makebox(0,0){$\bullet$}}
\put(28,1.2){\makebox(0,0){$\bullet$}}
\put(16,0.5){\line(1,0){12}}
\put(16,2.5){\line(1,0){12}}
\put(22,1.5){\makebox(0,0){$\langle$}}
\put(4,8){\makebox(0,0){$\scriptstyle 1$}}
\put(16,8){\makebox(0,0){$\scriptstyle 0$}}
\put(28,8){\makebox(0,0){$\scriptstyle 0$}}
\end{picture}\oplus\begin{picture}(32,5)
\put(4,1.5){\line(1,0){12}}
\put(4,1.2){\makebox(0,0){$\times$}}
\put(16,1.2){\makebox(0,0){$\bullet$}}
\put(28,1.2){\makebox(0,0){$\bullet$}}
\put(16,0.5){\line(1,0){12}}
\put(16,2.5){\line(1,0){12}}
\put(22,1.5){\makebox(0,0){$\langle$}}
\put(4,8){\makebox(0,0){$\scriptstyle -1$}}
\put(16,8){\makebox(0,0){$\scriptstyle 1$}}
\put(28,8){\makebox(0,0){$\scriptstyle 0$}}
\end{picture}\oplus\begin{picture}(32,5)
\put(4,1.5){\line(1,0){12}}
\put(4,1.2){\makebox(0,0){$\times$}}
\put(16,1.2){\makebox(0,0){$\bullet$}}
\put(28,1.2){\makebox(0,0){$\bullet$}}
\put(16,0.5){\line(1,0){12}}
\put(16,2.5){\line(1,0){12}}
\put(22,1.5){\makebox(0,0){$\langle$}}
\put(4,8){\makebox(0,0){$\scriptstyle -1$}}
\put(16,8){\makebox(0,0){$\scriptstyle 0$}}
\put(28,8){\makebox(0,0){$\scriptstyle 0$}}
\end{picture}$$
and our differential complex becomes
$$\begin{array}{cccccc}
{\mathbb{T}}&\hspace{-10pt}\to\hspace{-10pt}
&\Lambda^1\otimes{\mathbb{T}}
&\hspace{-10pt}\to\hspace{-10pt}
&\Lambda_\perp^2\otimes{\mathbb{T}}
&\hspace{-10pt}\to\cdots\\
\|&&\|&&\|\\[3pt]
\begin{array}{c}
\begin{picture}(32,5)
\put(4,1.5){\line(1,0){12}}
\put(4,1.2){\makebox(0,0){$\times$}}
\put(16,1.2){\makebox(0,0){$\bullet$}}
\put(28,1.2){\makebox(0,0){$\bullet$}}
\put(16,0.5){\line(1,0){12}}
\put(16,2.5){\line(1,0){12}}
\put(22,1.5){\makebox(0,0){$\langle$}}
\put(4,8){\makebox(0,0){$\scriptstyle 1$}}
\put(16,8){\makebox(0,0){$\scriptstyle 0$}}
\put(28,8){\makebox(0,0){$\scriptstyle 0$}}
\end{picture}\\ \oplus\\
\begin{picture}(32,5)
\put(4,1.5){\line(1,0){12}}
\put(4,1.2){\makebox(0,0){$\times$}}
\put(16,1.2){\makebox(0,0){$\bullet$}}
\put(28,1.2){\makebox(0,0){$\bullet$}}
\put(16,0.5){\line(1,0){12}}
\put(16,2.5){\line(1,0){12}}
\put(22,1.5){\makebox(0,0){$\langle$}}
\put(4,8){\makebox(0,0){$\scriptstyle -1$}}
\put(16,8){\makebox(0,0){$\scriptstyle 1$}}
\put(28,8){\makebox(0,0){$\scriptstyle 0$}}
\put(35,5){\vector(3,1){68}}
\end{picture}\\ \oplus\\
\begin{picture}(32,5)
\put(4,1.5){\line(1,0){12}}
\put(4,1.2){\makebox(0,0){$\times$}}
\put(16,1.2){\makebox(0,0){$\bullet$}}
\put(28,1.2){\makebox(0,0){$\bullet$}}
\put(16,0.5){\line(1,0){12}}
\put(16,2.5){\line(1,0){12}}
\put(22,1.5){\makebox(0,0){$\langle$}}
\put(4,8){\makebox(0,0){$\scriptstyle -1$}}
\put(16,8){\makebox(0,0){$\scriptstyle 0$}}
\put(28,8){\makebox(0,0){$\scriptstyle 0$}}
\put(35,5){\vector(3,2){25}}
\end{picture}
\end{array}
&&
\begin{array}{c}
\begin{picture}(32,5)
\put(4,1.5){\line(1,0){12}}
\put(4,1.2){\makebox(0,0){$\times$}}
\put(16,1.2){\makebox(0,0){$\bullet$}}
\put(28,1.2){\makebox(0,0){$\bullet$}}
\put(16,0.5){\line(1,0){12}}
\put(16,2.5){\line(1,0){12}}
\put(22,1.5){\makebox(0,0){$\langle$}}
\put(4,8){\makebox(0,0){$\scriptstyle -1$}}
\put(16,8){\makebox(0,0){$\scriptstyle 1$}}
\put(28,8){\makebox(0,0){$\scriptstyle 0$}}
\end{picture}\\ \oplus\\
\begin{picture}(32,5)
\put(4,1.5){\line(1,0){12}}
\put(4,1.2){\makebox(0,0){$\times$}}
\put(16,1.2){\makebox(0,0){$\bullet$}}
\put(28,1.2){\makebox(0,0){$\bullet$}}
\put(16,0.5){\line(1,0){12}}
\put(16,2.5){\line(1,0){12}}
\put(22,1.5){\makebox(0,0){$\langle$}}
\put(4,8){\makebox(0,0){$\scriptstyle -1$}}
\put(16,8){\makebox(0,0){$\scriptstyle 0$}}
\put(28,8){\makebox(0,0){$\scriptstyle 0$}}
\end{picture}\oplus\begin{picture}(32,5)
\put(4,1.5){\line(1,0){12}}
\put(4,1.2){\makebox(0,0){$\times$}}
\put(16,1.2){\makebox(0,0){$\bullet$}}
\put(28,1.2){\makebox(0,0){$\bullet$}}
\put(16,0.5){\line(1,0){12}}
\put(16,2.5){\line(1,0){12}}
\put(22,1.5){\makebox(0,0){$\langle$}}
\put(4,8){\makebox(0,0){$\scriptstyle -3$}}
\put(16,8){\makebox(0,0){$\scriptstyle 2$}}
\put(28,8){\makebox(0,0){$\scriptstyle 0$}}
\end{picture}\oplus\begin{picture}(32,5)
\put(4,1.5){\line(1,0){12}}
\put(4,1.2){\makebox(0,0){$\times$}}
\put(16,1.2){\makebox(0,0){$\bullet$}}
\put(28,1.2){\makebox(0,0){$\bullet$}}
\put(16,0.5){\line(1,0){12}}
\put(16,2.5){\line(1,0){12}}
\put(22,1.5){\makebox(0,0){$\langle$}}
\put(4,8){\makebox(0,0){$\scriptstyle -2$}}
\put(16,8){\makebox(0,0){$\scriptstyle 0$}}
\put(28,8){\makebox(0,0){$\scriptstyle 1$}}
\put(35,5){\vector(2,1){45}}
\end{picture}\\ \oplus\\
\begin{picture}(32,5)
\put(4,1.5){\line(1,0){12}}
\put(4,1.2){\makebox(0,0){$\times$}}
\put(16,1.2){\makebox(0,0){$\bullet$}}
\put(28,1.2){\makebox(0,0){$\bullet$}}
\put(16,0.5){\line(1,0){12}}
\put(16,2.5){\line(1,0){12}}
\put(22,1.5){\makebox(0,0){$\langle$}}
\put(4,8){\makebox(0,0){$\scriptstyle -3$}}
\put(16,8){\makebox(0,0){$\scriptstyle 1$}}
\put(28,8){\makebox(0,0){$\scriptstyle 0$}}
\put(35,5){\vector(3,1){70}}
\end{picture}
\end{array}
&&
\begin{array}{c}
\begin{picture}(32,5)
\put(4,1.5){\line(1,0){12}}
\put(4,1.2){\makebox(0,0){$\times$}}
\put(16,1.2){\makebox(0,0){$\bullet$}}
\put(28,1.2){\makebox(0,0){$\bullet$}}
\put(16,0.5){\line(1,0){12}}
\put(16,2.5){\line(1,0){12}}
\put(22,1.5){\makebox(0,0){$\langle$}}
\put(4,8){\makebox(0,0){$\scriptstyle -2$}}
\put(16,8){\makebox(0,0){$\scriptstyle 0$}}
\put(28,8){\makebox(0,0){$\scriptstyle 1$}}
\end{picture}\\ \oplus\\
\begin{picture}(32,5)
\put(4,1.5){\line(1,0){12}}
\put(4,1.2){\makebox(0,0){$\times$}}
\put(16,1.2){\makebox(0,0){$\bullet$}}
\put(28,1.2){\makebox(0,0){$\bullet$}}
\put(16,0.5){\line(1,0){12}}
\put(16,2.5){\line(1,0){12}}
\put(22,1.5){\makebox(0,0){$\langle$}}
\put(4,8){\makebox(0,0){$\scriptstyle -3$}}
\put(16,8){\makebox(0,0){$\scriptstyle 1$}}
\put(28,8){\makebox(0,0){$\scriptstyle 0$}}
\end{picture}\oplus\begin{picture}(32,5)
\put(4,1.5){\line(1,0){12}}
\put(4,1.2){\makebox(0,0){$\times$}}
\put(16,1.2){\makebox(0,0){$\bullet$}}
\put(28,1.2){\makebox(0,0){$\bullet$}}
\put(16,0.5){\line(1,0){12}}
\put(16,2.5){\line(1,0){12}}
\put(22,1.5){\makebox(0,0){$\langle$}}
\put(4,8){\makebox(0,0){$\scriptstyle -4$}}
\put(16,8){\makebox(0,0){$\scriptstyle 1$}}
\put(28,8){\makebox(0,0){$\scriptstyle 1$}}
\end{picture}\\ \oplus\\
\begin{picture}(32,5)
\put(4,1.5){\line(1,0){12}}
\put(4,1.2){\makebox(0,0){$\times$}}
\put(16,1.2){\makebox(0,0){$\bullet$}}
\put(28,1.2){\makebox(0,0){$\bullet$}}
\put(16,0.5){\line(1,0){12}}
\put(16,2.5){\line(1,0){12}}
\put(22,1.5){\makebox(0,0){$\langle$}}
\put(4,8){\makebox(0,0){$\scriptstyle -4$}}
\put(16,8){\makebox(0,0){$\scriptstyle 0$}}
\put(28,8){\makebox(0,0){$\scriptstyle 1$}}
\put(35,5){\vector(4,3){50}}
\end{picture}
\end{array}
\end{array}$$
in which all bundles have been decomposed into their irreducible parts and some
particular homomorphisms have been highlighted by diagonal arrows. For
$\nabla:{\mathbb{T}}\to\Lambda^1\otimes{\mathbb{T}}$, for example, we see 
from (\ref{tractor_connection}) that
$$\nabla_a\!\left[\!\begin{array}c0\\ \mu_b\\ \rho\end{array}\!\right]
\!\equiv\!
\left[\!\!\!\begin{array}c-\mu_a\\ 
\nabla_a\mu_b-J_{ab}\rho-J_{ab}\alpha^c\mu_c\\ 
\begin{array}{l}
\nabla_a\rho+\Rho_{ab}\mu^b-\frac{3}{2n-1}\Phi_{ab}\mu^b
\end{array}
\end{array}\!\!\!\right]$$
so
$$\Gamma(\begin{picture}(32,5)
\put(4,1.5){\line(1,0){12}}
\put(4,1.2){\makebox(0,0){$\times$}}
\put(16,1.2){\makebox(0,0){$\bullet$}}
\put(28,1.2){\makebox(0,0){$\bullet$}}
\put(16,0.5){\line(1,0){12}}
\put(16,2.5){\line(1,0){12}}
\put(22,1.5){\makebox(0,0){$\langle$}}
\put(4,8){\makebox(0,0){$\scriptstyle -1$}}
\put(16,8){\makebox(0,0){$\scriptstyle 1$}}
\put(28,8){\makebox(0,0){$\scriptstyle 0$}}
\end{picture})\ni\mu_b\longmapsto -\mu_a\in\Gamma(\begin{picture}(32,5)
\put(4,1.5){\line(1,0){12}}
\put(4,1.2){\makebox(0,0){$\times$}}
\put(16,1.2){\makebox(0,0){$\bullet$}}
\put(28,1.2){\makebox(0,0){$\bullet$}}
\put(16,0.5){\line(1,0){12}}
\put(16,2.5){\line(1,0){12}}
\put(22,1.5){\makebox(0,0){$\langle$}}
\put(4,8){\makebox(0,0){$\scriptstyle -1$}}
\put(16,8){\makebox(0,0){$\scriptstyle 1$}}
\put(28,8){\makebox(0,0){$\scriptstyle 0$}}
\end{picture})$$
and
$$\Gamma(\begin{picture}(32,5)
\put(4,1.5){\line(1,0){12}}
\put(4,1.2){\makebox(0,0){$\times$}}
\put(16,1.2){\makebox(0,0){$\bullet$}}
\put(28,1.2){\makebox(0,0){$\bullet$}}
\put(16,0.5){\line(1,0){12}}
\put(16,2.5){\line(1,0){12}}
\put(22,1.5){\makebox(0,0){$\langle$}}
\put(4,8){\makebox(0,0){$\scriptstyle -1$}}
\put(16,8){\makebox(0,0){$\scriptstyle 0$}}
\put(28,8){\makebox(0,0){$\scriptstyle 0$}}
\end{picture})\ni\rho\longmapsto-J_{ab}\rho\in\Gamma(\begin{picture}(32,5)
\put(4,1.5){\line(1,0){12}}
\put(4,1.2){\makebox(0,0){$\times$}}
\put(16,1.2){\makebox(0,0){$\bullet$}}
\put(28,1.2){\makebox(0,0){$\bullet$}}
\put(16,0.5){\line(1,0){12}}
\put(16,2.5){\line(1,0){12}}
\put(22,1.5){\makebox(0,0){$\langle$}}
\put(4,8){\makebox(0,0){$\scriptstyle -1$}}
\put(16,8){\makebox(0,0){$\scriptstyle 0$}}
\put(28,8){\makebox(0,0){$\scriptstyle 0$}}
\end{picture})\subset\Gamma(\Lambda^1\otimes\Lambda^1[1]).$$
Notice that all of these homomorphisms are, in fact, isomorphisms between 
the irreducible bundles involved. Diagram chasing now allows us to cancel 
these irreducible parts leaving a complex
having the same local cohomology as the full tractor-coupled Rumin--Seshadri
complex:
$$0\to\begin{picture}(32,5)
\put(4,1.5){\line(1,0){12}}
\put(4,1.2){\makebox(0,0){$\times$}}
\put(16,1.2){\makebox(0,0){$\bullet$}}
\put(28,1.2){\makebox(0,0){$\bullet$}}
\put(16,0.5){\line(1,0){12}}
\put(16,2.5){\line(1,0){12}}
\put(22,1.5){\makebox(0,0){$\langle$}}
\put(4,8){\makebox(0,0){$\scriptstyle 1$}}
\put(16,8){\makebox(0,0){$\scriptstyle 0$}}
\put(28,8){\makebox(0,0){$\scriptstyle 0$}}
\end{picture}\xrightarrow{\nabla^2}\begin{picture}(32,5)
\put(4,1.5){\line(1,0){12}}
\put(4,1.2){\makebox(0,0){$\times$}}
\put(16,1.2){\makebox(0,0){$\bullet$}}
\put(28,1.2){\makebox(0,0){$\bullet$}}
\put(16,0.5){\line(1,0){12}}
\put(16,2.5){\line(1,0){12}}
\put(22,1.5){\makebox(0,0){$\langle$}}
\put(4,8){\makebox(0,0){$\scriptstyle -3$}}
\put(16,8){\makebox(0,0){$\scriptstyle 2$}}
\put(28,8){\makebox(0,0){$\scriptstyle 0$}}
\end{picture}\xrightarrow{\nabla}\begin{picture}(32,5)
\put(4,1.5){\line(1,0){12}}
\put(4,1.2){\makebox(0,0){$\times$}}
\put(16,1.2){\makebox(0,0){$\bullet$}}
\put(28,1.2){\makebox(0,0){$\bullet$}}
\put(16,0.5){\line(1,0){12}}
\put(16,2.5){\line(1,0){12}}
\put(22,1.5){\makebox(0,0){$\langle$}}
\put(4,8){\makebox(0,0){$\scriptstyle -4$}}
\put(16,8){\makebox(0,0){$\scriptstyle 1$}}
\put(28,8){\makebox(0,0){$\scriptstyle 1$}}
\end{picture}\xrightarrow{\nabla^2}\begin{picture}(32,5)
\put(4,1.5){\line(1,0){12}}
\put(4,1.2){\makebox(0,0){$\times$}}
\put(16,1.2){\makebox(0,0){$\bullet$}}
\put(28,1.2){\makebox(0,0){$\bullet$}}
\put(16,0.5){\line(1,0){12}}
\put(16,2.5){\line(1,0){12}}
\put(22,1.5){\makebox(0,0){$\langle$}}
\put(4,8){\makebox(0,0){$\scriptstyle -6$}}
\put(16,8){\makebox(0,0){$\scriptstyle 1$}}
\put(28,8){\makebox(0,0){$\scriptstyle 1$}}
\end{picture}\xrightarrow{\nabla}\begin{picture}(32,5)
\put(4,1.5){\line(1,0){12}}
\put(4,1.2){\makebox(0,0){$\times$}}
\put(16,1.2){\makebox(0,0){$\bullet$}}
\put(28,1.2){\makebox(0,0){$\bullet$}}
\put(16,0.5){\line(1,0){12}}
\put(16,2.5){\line(1,0){12}}
\put(22,1.5){\makebox(0,0){$\langle$}}
\put(4,8){\makebox(0,0){$\scriptstyle -7$}}
\put(16,8){\makebox(0,0){$\scriptstyle 2$}}
\put(28,8){\makebox(0,0){$\scriptstyle 0$}}
\end{picture}\xrightarrow{\nabla^2}\begin{picture}(32,5)
\put(4,1.5){\line(1,0){12}}
\put(4,1.2){\makebox(0,0){$\times$}}
\put(16,1.2){\makebox(0,0){$\bullet$}}
\put(28,1.2){\makebox(0,0){$\bullet$}}
\put(16,0.5){\line(1,0){12}}
\put(16,2.5){\line(1,0){12}}
\put(22,1.5){\makebox(0,0){$\langle$}}
\put(4,8){\makebox(0,0){$\scriptstyle -7$}}
\put(16,8){\makebox(0,0){$\scriptstyle 0$}}
\put(28,8){\makebox(0,0){$\scriptstyle 0$}}
\end{picture}\to 0.$$
In this simple case it is straightforward to determine these operators
explicitly. For example, we see from (\ref{tractor_connection}) that the first
one is
$$\sigma\longmapsto
\nabla_a\nabla_b\sigma+\Rho_{ab}\sigma-\Phi_{ab}\sigma.$$
By construction, this operator is conformally invariant but we can also check 
this directly:
$$\begin{array}{rcl}
\hat\nabla_a\hat\nabla_b\sigma+\hat\Rho_{ab}\sigma-\hat\Phi_{ab}\sigma
&=&\nabla_a(\nabla_b\sigma+\Upsilon_b\sigma)
-\Upsilon_b(\nabla_a\sigma+\Upsilon_a\sigma)\\
&&\quad{}+(\Rho_{ab}-\nabla_a\Upsilon_b+\Upsilon_a\Upsilon_b)\sigma
-\Phi_{ab}\sigma\\[4pt]
&=&\nabla_a\nabla_b\sigma+\Rho_{ab}\sigma-\Phi_{ab}\sigma.\end{array}$$
The next operator $\nabla:\begin{picture}(32,5)
\put(4,1.5){\line(1,0){12}}
\put(4,1.2){\makebox(0,0){$\times$}}
\put(16,1.2){\makebox(0,0){$\bullet$}}
\put(28,1.2){\makebox(0,0){$\bullet$}}
\put(16,0.5){\line(1,0){12}}
\put(16,2.5){\line(1,0){12}}
\put(22,1.5){\makebox(0,0){$\langle$}}
\put(4,8){\makebox(0,0){$\scriptstyle -3$}}
\put(16,8){\makebox(0,0){$\scriptstyle 2$}}
\put(28,8){\makebox(0,0){$\scriptstyle 0$}}
\end{picture}\to\begin{picture}(32,5)
\put(4,1.5){\line(1,0){12}}
\put(4,1.2){\makebox(0,0){$\times$}}
\put(16,1.2){\makebox(0,0){$\bullet$}}
\put(28,1.2){\makebox(0,0){$\bullet$}}
\put(16,0.5){\line(1,0){12}}
\put(16,2.5){\line(1,0){12}}
\put(22,1.5){\makebox(0,0){$\langle$}}
\put(4,8){\makebox(0,0){$\scriptstyle -4$}}
\put(16,8){\makebox(0,0){$\scriptstyle 1$}}
\put(28,8){\makebox(0,0){$\scriptstyle 1$}}
\end{picture}$ is given by
$$\textstyle\psi_{ab}\longmapsto\nabla_{[a}\psi_{b]c}
-\frac1{2n+1}\left[J_{ab}J^{de}\nabla_d\psi_{ec}
-J^{de}(\nabla_d\psi_{e[a})J_{b]c}\right]\enskip
\mbox{for }\psi_{ab}=\psi_{(ab)}.$$

\begin{thm}\label{BGGtheorem}
Suppose $(M,[J,\nabla])$ is a conformally Fedosov manifold of dimension $2n$
whose invariant curvature $V_{abcd}$ vanishes. Then for any $n+1$ non-negative
integers $a,b,c,\cdots d,e$ there is a differential complex
\begin{equation}\label{BGGcomplex}\begin{array}{rll}\begin{picture}(72,5)
\put(4,1.5){\line(1,0){30}} 
\put(4,1.2){\makebox(0,0){$\times$}}
\put(16,1.2){\makebox(0,0){$\bullet$}} 
\put(28,1.2){\makebox(0,0){$\bullet$}}
\put(43,1.2){\makebox(0,0){$\cdots$}} 
\put(50,1.5){\line(1,0){6}}
\put(56,1.2){\makebox(0,0){$\bullet$}} 
\put(56,0.5){\line(1,0){12}}
\put(56,2.5){\line(1,0){12}} 
\put(62,1.5){\makebox(0,0){$\langle$}}
\put(68,1.2){\makebox(0,0){$\bullet$}} 
\put(4,6){\makebox(0,0)[b]{$\scriptstyle a$}} 
\put(16,6){\makebox(0,0)[b]{$\scriptstyle b$}}
\put(28,6){\makebox(0,0)[b]{$\scriptstyle c$}}
\put(56,6){\makebox(0,0)[b]{$\scriptstyle d$}}
\put(68,6){\makebox(0,0)[b]{$\scriptstyle e$}}
\end{picture}&\xrightarrow{\,\nabla^{a+1}\,}& 
\begin{picture}(102,5)
\put(4,1.5){\line(1,0){60}} 
\put(4,1.2){\makebox(0,0){$\times$}}
\put(36,1.2){\makebox(0,0){$\bullet$}} 
\put(58,1.2){\makebox(0,0){$\bullet$}}
\put(73,1.2){\makebox(0,0){$\cdots$}} 
\put(80,1.5){\line(1,0){6}}
\put(86,1.2){\makebox(0,0){$\bullet$}} 
\put(86,0.5){\line(1,0){12}}
\put(86,2.5){\line(1,0){12}} 
\put(92,1.5){\makebox(0,0){$\langle$}}
\put(98,1.2){\makebox(0,0){$\bullet$}} 
\put(4,5){\makebox(0,0)[b]{$\scriptstyle-a-2$}} 
\put(36,5){\makebox(0,0)[b]{$\scriptstyle a+b+1$}}
\put(58,6){\makebox(0,0)[b]{$\scriptstyle c$}}
\put(86,6){\makebox(0,0)[b]{$\scriptstyle d$}}
\put(98,6){\makebox(0,0)[b]{$\scriptstyle e$}} \end{picture}\\
&\quad\xrightarrow{\,\nabla^{b+1}\,}&
\quad\begin{picture}(97,5)
\put(4,1.5){\line(1,0){55}}
\put(4,1.2){\makebox(0,0){$\times$}}
\put(31,1.2){\makebox(0,0){$\bullet$}}
\put(53,1.2){\makebox(0,0){$\bullet$}}
\put(68,1.2){\makebox(0,0){$\cdots$}}
\put(75,1.5){\line(1,0){6}}
\put(81,1.2){\makebox(0,0){$\bullet$}}
\put(81,0.5){\line(1,0){12}}
\put(81,2.5){\line(1,0){12}}
\put(87,1.5){\makebox(0,0){$\langle$}}
\put(93,1.2){\makebox(0,0){$\bullet$}}
\put(4,5){\makebox(0,0)[b]{$\scriptstyle -a-b-3$}}
\put(31,6){\makebox(0,0)[b]{$\scriptstyle a$}}
\put(53,5){\makebox(0,0)[b]{$\scriptstyle b+c+1$}}
\put(81,6){\makebox(0,0)[b]{$\scriptstyle d$}}
\put(93,6){\makebox(0,0)[b]{$\scriptstyle e$}}
\end{picture}\\
&\qquad\makebox[0pt][l]{$\xrightarrow{\,\nabla^{c+1}\,}\enskip\cdots\,,$}
\end{array}\end{equation}
which is locally exact save at the zeroth and first positions, where its local 
cohomology may be identified with the locally constant sheaves 
\underbar{$\ker\Theta$} and \underbar{$\coker\Theta$}, respectively. Here, 
$$\Theta\in{\mathrm{Aut}}\Big(\begin{picture}(72,5)
\put(4,1.5){\line(1,0){30}}
\put(4,1.2){\makebox(0,0){$\bullet$}}
\put(16,1.2){\makebox(0,0){$\bullet$}}
\put(28,1.2){\makebox(0,0){$\bullet$}}
\put(43,1.2){\makebox(0,0){$\cdots$}}
\put(50,1.5){\line(1,0){6}}
\put(56,1.2){\makebox(0,0){$\bullet$}}
\put(56,0.5){\line(1,0){12}}
\put(56,2.5){\line(1,0){12}}
\put(62,1.5){\makebox(0,0){$\langle$}}
\put(68,1.2){\makebox(0,0){$\bullet$}}
\put(4,6){\makebox(0,0)[b]{$\scriptstyle a$}}
\put(16,6){\makebox(0,0)[b]{$\scriptstyle b$}}
\put(28,6){\makebox(0,0)[b]{$\scriptstyle c$}}
\put(56,6){\makebox(0,0)[b]{$\scriptstyle d$}}
\put(68,6){\makebox(0,0)[b]{$\scriptstyle e$}}
\end{picture}\big({\mathbb{T}}\big)\Big)$$
is induced by $\Theta:{\mathbb{T}}\to{\mathbb{T}}$ from \eqref{Einstein} and 
$\begin{picture}(72,5)
\put(4,1.5){\line(1,0){30}}
\put(4,1.2){\makebox(0,0){$\bullet$}}
\put(16,1.2){\makebox(0,0){$\bullet$}}
\put(28,1.2){\makebox(0,0){$\bullet$}}
\put(43,1.2){\makebox(0,0){$\cdots$}}
\put(50,1.5){\line(1,0){6}}
\put(56,1.2){\makebox(0,0){$\bullet$}}
\put(56,0.5){\line(1,0){12}}
\put(56,2.5){\line(1,0){12}}
\put(62,1.5){\makebox(0,0){$\langle$}}
\put(68,1.2){\makebox(0,0){$\bullet$}}
\put(4,6){\makebox(0,0)[b]{$\scriptstyle a$}}
\put(16,6){\makebox(0,0)[b]{$\scriptstyle b$}}
\put(28,6){\makebox(0,0)[b]{$\scriptstyle c$}}
\put(56,6){\makebox(0,0)[b]{$\scriptstyle d$}}
\put(68,6){\makebox(0,0)[b]{$\scriptstyle e$}}
\end{picture}\big({\mathbb{T}}\big)$
is the bundle associated to ${\mathbb{T}}$ via the 
${\mathrm{Sp}}(2n+2,{\mathbb{R}})$-module
$\begin{picture}(72,5)
\put(4,1.5){\line(1,0){30}}
\put(4,1.2){\makebox(0,0){$\bullet$}}
\put(16,1.2){\makebox(0,0){$\bullet$}}
\put(28,1.2){\makebox(0,0){$\bullet$}}
\put(43,1.2){\makebox(0,0){$\cdots$}}
\put(50,1.5){\line(1,0){6}}
\put(56,1.2){\makebox(0,0){$\bullet$}}
\put(56,0.5){\line(1,0){12}}
\put(56,2.5){\line(1,0){12}}
\put(62,1.5){\makebox(0,0){$\langle$}}
\put(68,1.2){\makebox(0,0){$\bullet$}}
\put(4,6){\makebox(0,0)[b]{$\scriptstyle a$}}
\put(16,6){\makebox(0,0)[b]{$\scriptstyle b$}}
\put(28,6){\makebox(0,0)[b]{$\scriptstyle c$}}
\put(56,6){\makebox(0,0)[b]{$\scriptstyle d$}}
\put(68,6){\makebox(0,0)[b]{$\scriptstyle e$}}
\end{picture}$, 
bearing in mind that the non-degenerate skew form \eqref{bigJ} reduces the
structure group of\/ ${\mathbb{T}}$ to ${\mathrm{Sp}}(2n+2,{\mathbb{R}})$.
\end{thm}
\begin{proof}
One considers the coupled 
Rumin-Seshadri complex (\ref{coupledweightedRS}) with $E=\begin{picture}(72,5)
\put(4,1.5){\line(1,0){30}}
\put(4,1.2){\makebox(0,0){$\bullet$}}
\put(16,1.2){\makebox(0,0){$\bullet$}}
\put(28,1.2){\makebox(0,0){$\bullet$}}
\put(43,1.2){\makebox(0,0){$\cdots$}}
\put(50,1.5){\line(1,0){6}}
\put(56,1.2){\makebox(0,0){$\bullet$}}
\put(56,0.5){\line(1,0){12}}
\put(56,2.5){\line(1,0){12}}
\put(62,1.5){\makebox(0,0){$\langle$}}
\put(68,1.2){\makebox(0,0){$\bullet$}}
\put(4,6){\makebox(0,0)[b]{$\scriptstyle a$}}
\put(16,6){\makebox(0,0)[b]{$\scriptstyle b$}}
\put(28,6){\makebox(0,0)[b]{$\scriptstyle c$}}
\put(56,6){\makebox(0,0)[b]{$\scriptstyle d$}}
\put(68,6){\makebox(0,0)[b]{$\scriptstyle e$}}
\end{picture}\big({\mathbb{T}}\big)$ and~$w=-2$. 
According to Corollary~\ref{maincorollary} and Theorem~\ref{coupledRStheorem},
this complex is locally exact save at the zeroth and first positions, where its
local cohomology may be identified with \underbar{$\ker\Theta$} and
\underbar{$\coker\Theta$}, respectively. We assert that, as in the example
discussed at the beginning of this section, there are algebraic cancellations
resulting in the complex~(\ref{BGGcomplex}). These automatic cancellations are 
a consequence of Kostant's theorem \cite{K} on Lie algebra cohomology as is now
familiar in parabolic geometry~\cite{CS}.
\end{proof}

\paragraph{\bf Remarks} This method of proof is often referred to as employing
the `BGG-machinery' where BGG stands for Bernstein--Gelfand--Gelfand in
reference to~\cite{BGG} and~\cite{Lepowsky}, where dual complexes are 
constructed on the
level of induced modules in representation theory. The complexes
(\ref{BGGcomplex}) lie outside the parabolic realm but follow exactly the BGG
complexes on the contact projective sphere~\cite[\S1.1.4]{CS}. More precisely,
the sphere $S^{2n+1}$ as a homogeneous space for
${\mathrm{Sp}}(2n+2,{\mathbb{R}})$, is the flat model for a type of parabolic
geometry known as {\em contact projective\/}~\cite{F}. A link between these 
constructions is provided by the fibration 
$$S^{2n+1}\to{\mathbb{CP}}_n$$
and, following~\cite{capSalac} and bearing in mind that ${\mathbb{CP}}_n$ is
our basic example of a conformally Fedosov manifold with $V_{abcd}=0$, the BGG
complexes (\ref{BGGcomplex}) may be seen as symmetry reductions of the usual
BGG complexes on the contact projective sphere $S^{2n+1}$ as a parabolic
geometry. 
We should also mention that the authors of \cite{capSalac} are currently
preparing an article based on~\cite{capSrni14}, which further applies the
symmetry reduction technique of \cite{capSalac} to the other contact parabolic
geometries (as listed in \cite[\S4.2]{CS}), and which we expect also applies to
curved conformally Fedosov structures, as suitable symmetry reductions of
curved contact projective structures.

In any case, the initial portion 
$$\begin{picture}(72,5)
\put(4,1.5){\line(1,0){30}}
\put(4,1.2){\makebox(0,0){$\times$}}
\put(16,1.2){\makebox(0,0){$\bullet$}}
\put(28,1.2){\makebox(0,0){$\bullet$}}
\put(43,1.2){\makebox(0,0){$\cdots$}}
\put(50,1.5){\line(1,0){6}}
\put(56,1.2){\makebox(0,0){$\bullet$}}
\put(56,0.5){\line(1,0){12}}
\put(56,2.5){\line(1,0){12}}
\put(62,1.5){\makebox(0,0){$\langle$}}
\put(68,1.2){\makebox(0,0){$\bullet$}}
\put(4,6){\makebox(0,0)[b]{$\scriptstyle 0$}}
\put(16,6){\makebox(0,0)[b]{$\scriptstyle 1$}}
\put(28,6){\makebox(0,0)[b]{$\scriptstyle 0$}}
\put(56,6){\makebox(0,0)[b]{$\scriptstyle 0$}}
\put(68,6){\makebox(0,0)[b]{$\scriptstyle 0$}}
\end{picture}\xrightarrow{\,\nabla\,}\enskip\begin{picture}(72,5)
\put(4,1.5){\line(1,0){30}}
\put(4,1.2){\makebox(0,0){$\times$}}
\put(16,1.2){\makebox(0,0){$\bullet$}}
\put(28,1.2){\makebox(0,0){$\bullet$}}
\put(43,1.2){\makebox(0,0){$\cdots$}}
\put(50,1.5){\line(1,0){6}}
\put(56,1.2){\makebox(0,0){$\bullet$}}
\put(56,0.5){\line(1,0){12}}
\put(56,2.5){\line(1,0){12}}
\put(62,1.5){\makebox(0,0){$\langle$}}
\put(68,1.2){\makebox(0,0){$\bullet$}}
\put(2,5.5){\makebox(0,0)[b]{$\scriptstyle -2$}}
\put(16,6){\makebox(0,0)[b]{$\scriptstyle 2$}}
\put(28,6){\makebox(0,0)[b]{$\scriptstyle 0$}}
\put(56,6){\makebox(0,0)[b]{$\scriptstyle 0$}}
\put(68,6){\makebox(0,0)[b]{$\scriptstyle 0$}}
\end{picture}\xrightarrow{\,\nabla^2\,}\enskip\begin{picture}(72,5)
\put(4,1.5){\line(1,0){30}}
\put(4,1.2){\makebox(0,0){$\times$}}
\put(16,1.2){\makebox(0,0){$\bullet$}}
\put(28,1.2){\makebox(0,0){$\bullet$}}
\put(43,1.2){\makebox(0,0){$\cdots$}}
\put(50,1.5){\line(1,0){6}}
\put(56,1.2){\makebox(0,0){$\bullet$}}
\put(56,0.5){\line(1,0){12}}
\put(56,2.5){\line(1,0){12}}
\put(62,1.5){\makebox(0,0){$\langle$}}
\put(68,1.2){\makebox(0,0){$\bullet$}}
\put(2,5.5){\makebox(0,0)[b]{$\scriptstyle -4$}}
\put(16,6){\makebox(0,0)[b]{$\scriptstyle 0$}}
\put(28,6){\makebox(0,0)[b]{$\scriptstyle 2$}}
\put(56,6){\makebox(0,0)[b]{$\scriptstyle 0$}}
\put(68,6){\makebox(0,0)[b]{$\scriptstyle 0$}}
\end{picture}$$
of a BGG complex (\ref{BGGcomplex}) on ${\mathbb{CP}}_n$ is the subject of
\cite[Theorem~8]{EG} where it is shown that the second operator provides
exactly the integrability conditions for the range of the Killing operator
on~${\mathbb{CP}}_n$. This conclusion is immediate from
Theorem~\ref{BGGtheorem}: since ${\mathbb{CP}}_n$ is simply-connected, there is
no global cohomology arising from~\underbar{$\coker\Theta$}.

The initial portion
$$\begin{picture}(82,5)
\put(4,1.5){\line(1,0){40}}
\put(4,1.2){\makebox(0,0){$\times$}}
\put(21,1.2){\makebox(0,0){$\bullet$}}
\put(38,1.2){\makebox(0,0){$\bullet$}}
\put(53,1.2){\makebox(0,0){$\cdots$}}
\put(60,1.5){\line(1,0){6}}
\put(66,1.2){\makebox(0,0){$\bullet$}}
\put(66,0.5){\line(1,0){12}}
\put(66,2.5){\line(1,0){12}}
\put(72,1.5){\makebox(0,0){$\langle$}}
\put(78,1.2){\makebox(0,0){$\bullet$}}
\put(4,6){\makebox(0,0)[b]{$\scriptstyle 0$}}
\put(21,6){\makebox(0,0)[b]{$\scriptstyle \ell-1$}}
\put(38,6){\makebox(0,0)[b]{$\scriptstyle 0$}}
\put(66,6){\makebox(0,0)[b]{$\scriptstyle 0$}}
\put(78,6){\makebox(0,0)[b]{$\scriptstyle 0$}}
\end{picture}\xrightarrow{\,\nabla\,}\enskip
\begin{picture}(72,5)
\put(4,1.5){\line(1,0){30}}
\put(4,1.2){\makebox(0,0){$\times$}}
\put(16,1.2){\makebox(0,0){$\bullet$}}
\put(28,1.2){\makebox(0,0){$\bullet$}}
\put(43,1.2){\makebox(0,0){$\cdots$}}
\put(50,1.5){\line(1,0){6}}
\put(56,1.2){\makebox(0,0){$\bullet$}}
\put(56,0.5){\line(1,0){12}}
\put(56,2.5){\line(1,0){12}}
\put(62,1.5){\makebox(0,0){$\langle$}}
\put(68,1.2){\makebox(0,0){$\bullet$}}
\put(2,5.5){\makebox(0,0)[b]{$\scriptstyle -2$}}
\put(16,6){\makebox(0,0)[b]{$\scriptstyle \ell$}}
\put(28,6){\makebox(0,0)[b]{$\scriptstyle 0$}}
\put(56,6){\makebox(0,0)[b]{$\scriptstyle 0$}}
\put(68,6){\makebox(0,0)[b]{$\scriptstyle 0$}}
\end{picture}\xrightarrow{\,\nabla^2\,}\quad\begin{picture}(77,5)
\put(4,1.5){\line(1,0){35}}
\put(4,1.2){\makebox(0,0){$\times$}}
\put(21,1.2){\makebox(0,0){$\bullet$}}
\put(33,1.2){\makebox(0,0){$\bullet$}}
\put(48,1.2){\makebox(0,0){$\cdots$}}
\put(55,1.5){\line(1,0){6}}
\put(61,1.2){\makebox(0,0){$\bullet$}}
\put(61,0.5){\line(1,0){12}}
\put(61,2.5){\line(1,0){12}}
\put(67,1.5){\makebox(0,0){$\langle$}}
\put(73,1.2){\makebox(0,0){$\bullet$}}
\put(0,5.5){\makebox(0,0)[b]{$\scriptstyle -\ell-2$}}
\put(21,6){\makebox(0,0)[b]{$\scriptstyle 0$}}
\put(33,6){\makebox(0,0)[b]{$\scriptstyle \ell$}}
\put(61,6){\makebox(0,0)[b]{$\scriptstyle 0$}}
\put(73,6){\makebox(0,0)[b]{$\scriptstyle 0$}}
\end{picture}$$
of (\ref{BGGcomplex}) on ${\mathbb{CP}}_n$ is the subject
of~\cite[Theorem~9]{EG}.

Finally, we recall that in ordinary parabolic geometry, there are, not only
`BGG complexes' on the flat model $G/P$, as presented in~\cite{BE}, but also
`BGG sequences' on the corresponding curved geometries, as constructed
in~\cite{CD,CSS}. We anticipate BGG sequences in the conformally 
Fedosov setting and also that their construction follow most easily the route 
recently presented in~\cite{CS1,CS2}.

%

\end{document}